\documentclass[11pt,reqno]{amsart}
\usepackage[letterpaper,margin=1in]{geometry}
\usepackage[T1]{fontenc}
\usepackage{lmodern}
\usepackage[utf8]{inputenc}
\usepackage{amsmath,amssymb,amsthm,mathtools}
\usepackage{booktabs,array}
\usepackage{tikz}
\usetikzlibrary{arrows.meta,calc,decorations.markings,decorations.pathreplacing,positioning}
\usepackage{graphicx}
\usepackage{subcaption}
\usepackage{xcolor}
\usepackage{listings}
\usepackage{microtype}
\usepackage{mathrsfs}

\newtheorem{theorem}{Theorem}[section]
\newtheorem{proposition}[theorem]{Proposition}

\newtheorem{lemma}[theorem]{Lemma}
\newtheorem{corollary}[theorem]{Corollary}
\theoremstyle{definition}
\newtheorem{definition}[theorem]{Definition}
\newtheorem{example}[theorem]{Example}
\theoremstyle{remark}
\newtheorem{remark}[theorem]{Remark}
\newtheorem{problem}[theorem]{Problem}
\newcommand{\F}{\mathbb F_2}
\newcommand{\cC}{\mathcal C}
\newcommand{\cB}{\mathcal B}

\newcommand{\rank}{\operatorname{rank}}
\newcommand{\nul}{\operatorname{nullity}}

\definecolor{codegray}{gray}{0.96}
\definecolor{orbitblue}{RGB}{38,126,190}
\definecolor{orbitred}{RGB}{220,45,48}
\definecolor{orbitgreen}{RGB}{43,160,44}
\definecolor{orbitorange}{RGB}{255,127,14}
\definecolor{orbitpurple}{RGB}{152,78,163}
\usepackage[hidelinks]{hyperref}
\numberwithin{equation}{section}
\numberwithin{figure}{section}
\title[Billiard orbits in Young diagrams]
{Billiard Orbits in Young Diagrams: Medial Links, Bicycle Spaces, and Domino Tilings}

\author{David J. Hemmer}
\address{Department of Mathematical Sciences\\
  Michigan Technological University\\
  Houghton, MI 49931}
\email{djhemmer@mtu.edu}
\date{August 2026}

\begin{document}
\begin{abstract}
We study diagonal billiard trajectories inside the Young diagram of an
integer partition $\lambda$. A trajectory has slope $\pm1$, passes straight
through sides shared by adjacent cells, and reflects from the exterior
boundary until it closes. Let $\sigma(\lambda)$ denote the number of
resulting closed orbits. Motivated in part by a talk of Botkin and Schneider
on integer partitions and geometric diagrams arising from African and Indian
mathematical traditions, this construction extends the mirror-curve model of
Chokwe \emph{sona} sand drawings studied by Gerdes from rectangular grids to
arbitrary Young diagrams.

We use graph theory to reduce the problem of counting orbits to linear algebra. Let $G_\lambda$ be the cell-adjacency graph of $\lambda$ with its natural
planar embedding. We identify the billiard orbits with the components of
the medial link of $G_\lambda$ and deduce
\[
\sigma(\lambda)
=
1+\dim_{\mathbb F_2}\mathcal B(G_\lambda)
=
\operatorname{nullity}_{\mathbb F_2}L(G_\lambda),
\]
where $\mathcal B(G_\lambda)$ is the binary bicycle space and
$L(G_\lambda)$ is the mod-$2$ Laplacian. If $\lambda^\square$ is obtained
by deleting the first row and first column of $\lambda$, we further prove
\[
\sigma(\lambda)
=
1+\operatorname{nullity}_{\mathbb F_2}
A(G_{\lambda^\square}).
\]
For rectangular diagrams this recovers Gerdes' formula $\sigma(n^m)=\gcd(m,n)$
through identities for Fibonacci polynomials over $\mathbb F_2$. It also
gives the characterization
\[
\sigma(\lambda)=1
\quad\Longleftrightarrow\quad
\lambda^\square
\text{ has an odd number of domino tilings}.
\]

Using the checkerboard bipartition of $\lambda^\square$, we decompose
$\sigma(\lambda)-1$ into a color-imbalance term (related to the BG-rank of Berkovich–Garvan) and an even
rank-deficiency term. This yields parity restrictions and lower bounds
for the orbit number. It also shows that for any fixed $d$, asymptotically all partitions have more than $d$ orbits. We also show that $\sigma(\lambda)$ is at most the Durfee length of $\lambda$.
We  determine for staircase partitions
\[
\sigma(n,n-1,\ldots,1)
=
\left\lceil\frac{n}{2}\right\rceil.
\]
Finally we show 
the adjacency-nullity formula is independent of the ground field.
\end{abstract}

\maketitle

\section{Introduction}

The billiard construction considered here is inspired by the
\emph{sona} sand-drawing tradition of the Chokwe people of Angola and
neighboring regions and the related traditional floor-drawing practices from India. In some designs, a storyteller marks out a
rectangular array of points and then draws one or more closed curves
winding around the points according to geometric rules. A particularly
simple family, often called \emph{plaited-mat} designs, can be modeled by
a ray traveling at angle $45^\circ$ inside a rectangle and reflecting
from its sides; after the sharp reflections are rounded, one obtains the
familiar woven appearance of the sona. This ``mirror-curve'' interpretation was developed mathematically by
Paulus Gerdes and others; see
\cite{Gerdes1999,vitturi2006sona}. In particular, Gerdes observed that
the number of closed curves in an $m\times n$ rectangular grid is
$\gcd(m,n)$.
Gerdes also considers the effect of placing small double-sided mirrors inside the rectangle.

In a talk of Botkin and Schneider \cite{BotkinSchneider}, the idea was presented to generalize this construction to Young diagrams of arbitrary partition shape, which one could consider as having the entire boundary of the Young diagram made up of mirrors. Related knot-theoretic questions were earlier investigated by Harp and Schneider \cite{HarpSchneider2019}.

Let $\lambda=(\lambda_1,\ldots,\lambda_\ell)$
be a partition, with
$\lambda_1\geq\lambda_2\geq\cdots\geq\lambda_\ell>0.$
We identify its Young diagram with the set
\[
[\lambda]
=
\{(i,j):1\leq i\leq\ell,\ 1\leq j\leq\lambda_i\}.
\]
The box $(i,j)$ is the unit square in row $i$ and column $j$. We often refer to $\lambda$ interchangeably with its Young diagram.

Inside the Young diagram, consider diagonal trajectories of slope
$\pm1$. A trajectory passes straight through a side shared by two boxes
and reflects from a side on the exterior boundary. Starting from the
midpoint of any exterior unit edge produces a periodic trajectory, and
the exterior edge midpoints are partitioned into a finite collection of
such trajectories. We denote the number of trajectories by
$\sigma(\lambda)$. Figure~\ref{fig:orbitexamples} gives several examples
with varying numbers of trajectories, which we may also call orbits. Crossings at shared side midpoints are
not considered junctions: the two strands pass straight through one
another.

\begin{figure}[ht]
\centering
\begin{subfigure}[t]{0.23\textwidth}
  \centering
\begin{tikzpicture}[x=0.68cm,y=0.68cm]
\draw[black!75,line width=0.35pt] (0,2) rectangle (1,3);
\draw[black!75,line width=0.35pt] (1,2) rectangle (2,3);
\draw[black!75,line width=0.35pt] (2,2) rectangle (3,3);
\draw[black!75,line width=0.35pt] (3,2) rectangle (4,3);
\draw[black!75,line width=0.35pt] (0,1) rectangle (1,2);
\draw[black!75,line width=0.35pt] (1,1) rectangle (2,2);
\draw[black!75,line width=0.35pt] (2,1) rectangle (3,2);
\draw[black!75,line width=0.35pt] (0,0) rectangle (1,1);
\draw[orbitblue,line width=1.4pt,line cap=round,line join=round]
(0.000,1.500) -- (0.500,1.000) -- (1.000,0.500) -- (0.500,0.000) -- (0.000,0.500) -- (0.500,1.000) -- (1.000,1.500) -- (1.500,2.000) -- (2.000,2.500) -- (2.500,3.000) -- (3.000,2.500) -- (3.500,2.000) -- (4.000,2.500) -- (3.500,3.000) -- (3.000,2.500) -- (2.500,2.000) -- (2.000,1.500) -- (1.500,1.000) -- (1.000,1.500) -- (0.500,2.000) -- (0.000,2.500) -- (0.500,3.000) -- (1.000,2.500) -- (1.500,2.000) -- (2.000,1.500) -- (2.500,1.000) -- (3.000,1.500) -- (2.500,2.000) -- (2.000,2.500) -- (1.500,3.000) -- (1.000,2.500) -- (0.500,2.000) -- (0.000,1.500);
\end{tikzpicture}
  \caption{$\lambda=(4,3,1)$, $\sigma=1$}
\end{subfigure}\hfill
\begin{subfigure}[t]{0.23\textwidth}
  \centering
\begin{tikzpicture}[x=0.68cm,y=0.68cm]
\draw[black!75,line width=0.35pt] (0,2) rectangle (1,3);
\draw[black!75,line width=0.35pt] (1,2) rectangle (2,3);
\draw[black!75,line width=0.35pt] (2,2) rectangle (3,3);
\draw[black!75,line width=0.35pt] (3,2) rectangle (4,3);
\draw[black!75,line width=0.35pt] (4,2) rectangle (5,3);
\draw[black!75,line width=0.35pt] (0,1) rectangle (1,2);
\draw[black!75,line width=0.35pt] (1,1) rectangle (2,2);
\draw[black!75,line width=0.35pt] (0,0) rectangle (1,1);
\draw[orbitblue,line width=1.4pt,line cap=round,line join=round]
(0.500,3.000) -- (0.000,2.500) -- (0.500,2.000) -- (1.000,1.500) -- (1.500,1.000) -- (2.000,1.500) -- (1.500,2.000) -- (1.000,2.500) -- (0.500,3.000);
\draw[orbitred,line width=1.4pt,line cap=round,line join=round]
(0.000,1.500) -- (0.500,1.000) -- (1.000,0.500) -- (0.500,0.000) -- (0.000,0.500) -- (0.500,1.000) -- (1.000,1.500) -- (1.500,2.000) -- (2.000,2.500) -- (2.500,3.000) -- (3.000,2.500) -- (3.500,2.000) -- (4.000,2.500) -- (4.500,3.000) -- (5.000,2.500) -- (4.500,2.000) -- (4.000,2.500) -- (3.500,3.000) -- (3.000,2.500) -- (2.500,2.000) -- (2.000,2.500) -- (1.500,3.000) -- (1.000,2.500) -- (0.500,2.000) -- (0.000,1.500);
\end{tikzpicture}
  \caption{$\lambda=(5,2,1)$, $\sigma=2$}
\end{subfigure}\hfill
\begin{subfigure}[t]{0.23\textwidth}
  \centering
\begin{tikzpicture}[x=0.68cm,y=0.68cm]
\draw[black!75,line width=0.35pt] (0,2) rectangle (1,3);
\draw[black!75,line width=0.35pt] (1,2) rectangle (2,3);
\draw[black!75,line width=0.35pt] (2,2) rectangle (3,3);
\draw[black!75,line width=0.35pt] (3,2) rectangle (4,3);
\draw[black!75,line width=0.35pt] (0,1) rectangle (1,2);
\draw[black!75,line width=0.35pt] (1,1) rectangle (2,2);
\draw[black!75,line width=0.35pt] (2,1) rectangle (3,2);
\draw[black!75,line width=0.35pt] (0,0) rectangle (1,1);
\draw[black!75,line width=0.35pt] (1,0) rectangle (2,1);
\draw[black!75,line width=0.35pt] (2,0) rectangle (3,1);
\draw[orbitblue,line width=1.4pt,line cap=round,line join=round]
(0.500,3.000) -- (0.000,2.500) -- (0.500,2.000) -- (1.000,1.500) -- (1.500,1.000) -- (2.000,0.500) -- (2.500,0.000) -- (3.000,0.500) -- (2.500,1.000) -- (2.000,1.500) -- (1.500,2.000) -- (1.000,2.500) -- (0.500,3.000);
\draw[orbitred,line width=1.4pt,line cap=round,line join=round]
(0.000,1.500) -- (0.500,1.000) -- (1.000,0.500) -- (1.500,0.000) -- (2.000,0.500) -- (2.500,1.000) -- (3.000,1.500) -- (2.500,2.000) -- (2.000,2.500) -- (1.500,3.000) -- (1.000,2.500) -- (0.500,2.000) -- (0.000,1.500);
\draw[orbitgreen,line width=1.4pt,line cap=round,line join=round]
(0.000,0.500) -- (0.500,0.000) -- (1.000,0.500) -- (1.500,1.000) -- (2.000,1.500) -- (2.500,2.000) -- (3.000,2.500) -- (3.500,3.000) -- (4.000,2.500) -- (3.500,2.000) -- (3.000,2.500) -- (2.500,3.000) -- (2.000,2.500) -- (1.500,2.000) -- (1.000,1.500) -- (0.500,1.000) -- (0.000,0.500);
\end{tikzpicture}
  \caption{$\lambda=(4,3,3)$, $\sigma=3$}
\end{subfigure}\hfill
\begin{subfigure}[t]{0.23\textwidth}
  \centering
\begin{tikzpicture}[x=0.68cm,y=0.68cm]
\draw[black!75,line width=0.35pt] (0,3) rectangle (1,4);
\draw[black!75,line width=0.35pt] (1,3) rectangle (2,4);
\draw[black!75,line width=0.35pt] (2,3) rectangle (3,4);
\draw[black!75,line width=0.35pt] (3,3) rectangle (4,4);
\draw[black!75,line width=0.35pt] (4,3) rectangle (5,4);
\draw[black!75,line width=0.35pt] (0,2) rectangle (1,3);
\draw[black!75,line width=0.35pt] (1,2) rectangle (2,3);
\draw[black!75,line width=0.35pt] (2,2) rectangle (3,3);
\draw[black!75,line width=0.35pt] (3,2) rectangle (4,3);
\draw[black!75,line width=0.35pt] (0,1) rectangle (1,2);
\draw[black!75,line width=0.35pt] (1,1) rectangle (2,2);
\draw[black!75,line width=0.35pt] (2,1) rectangle (3,2);
\draw[black!75,line width=0.35pt] (3,1) rectangle (4,2);
\draw[black!75,line width=0.35pt] (0,0) rectangle (1,1);
\draw[black!75,line width=0.35pt] (1,0) rectangle (2,1);
\draw[black!75,line width=0.35pt] (2,0) rectangle (3,1);
\draw[black!75,line width=0.35pt] (3,0) rectangle (4,1);
\draw[orbitblue,line width=1.4pt,line cap=round,line join=round]
(0.500,4.000) -- (0.000,3.500) -- (0.500,3.000) -- (1.000,2.500) -- (1.500,2.000) -- (2.000,1.500) -- (2.500,1.000) -- (3.000,0.500) -- (3.500,0.000) -- (4.000,0.500) -- (3.500,1.000) -- (3.000,1.500) -- (2.500,2.000) -- (2.000,2.500) -- (1.500,3.000) -- (1.000,3.500) -- (0.500,4.000);
\draw[orbitred,line width=1.4pt,line cap=round,line join=round]
(2.500,0.000) -- (2.000,0.500) -- (1.500,1.000) -- (1.000,1.500) -- (0.500,2.000) -- (0.000,2.500) -- (0.500,3.000) -- (1.000,3.500) -- (1.500,4.000) -- (2.000,3.500) -- (2.500,3.000) -- (3.000,2.500) -- (3.500,2.000) -- (4.000,1.500) -- (3.500,1.000) -- (3.000,0.500) -- (2.500,0.000);
\draw[orbitgreen,line width=1.4pt,line cap=round,line join=round]
(0.000,1.500) -- (0.500,1.000) -- (1.000,0.500) -- (1.500,0.000) -- (2.000,0.500) -- (2.500,1.000) -- (3.000,1.500) -- (3.500,2.000) -- (4.000,2.500) -- (3.500,3.000) -- (3.000,3.500) -- (2.500,4.000) -- (2.000,3.500) -- (1.500,3.000) -- (1.000,2.500) -- (0.500,2.000) -- (0.000,1.500);
\draw[orbitorange,line width=1.4pt,line cap=round,line join=round]
(0.000,0.500) -- (0.500,0.000) -- (1.000,0.500) -- (1.500,1.000) -- (2.000,1.500) -- (2.500,2.000) -- (3.000,2.500) -- (3.500,3.000) -- (4.000,3.500) -- (4.500,4.000) -- (5.000,3.500) -- (4.500,3.000) -- (4.000,3.500) -- (3.500,4.000) -- (3.000,3.500) -- (2.500,3.000) -- (2.000,2.500) -- (1.500,2.000) -- (1.000,1.500) -- (0.500,1.000) -- (0.000,0.500);
\end{tikzpicture}
  \caption{$\lambda=(5,4,4,4)$, $\sigma=4$}
\end{subfigure}
\caption{Billiard trajectories for several partitions $\lambda$.}
\label{fig:orbitexamples}
\end{figure}
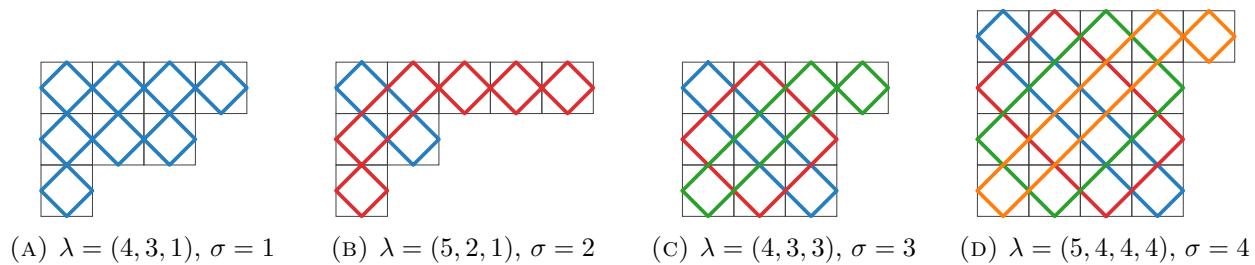

For a partition $\lambda$, the conjugate or transpose partition
$\lambda'$ is obtained by reflecting the Young diagram of $\lambda$
across its main diagonal, or equivalently by interchanging rows and
columns. This clearly preserves the number of trajectories:

\begin{lemma}
\label{lem:conjugation-invariance}
Let $\lambda$ be a partition. Then $\sigma(\lambda)=\sigma(\lambda').$
\end{lemma}

We first identify the billiard curves with the medial link of the cell-adjacency graph and express their number through the bicycle space and mod-$2$ Laplacian. We then reduce the computation to the adjacency matrix of the truncated diagram $\lambda^\square$. Using this formula and properties of Fibonacci polynomials, we can rederive the gcd formula for rectangular shapes. We also characterize one-orbit partitions by the parity of domino tilings. Finally, we develop ``checkerboard bounds", prove a Durfee-square upper bound, determine the orbit count for staircase partitions, establish an integral refinement, and give some initial enumeration results.

\section{Graph-theoretic description of $\sigma(\lambda)$}
\label{sec:graph-theory}
In this section we show that determining $\sigma(\lambda)$ can be reduced, via graph theory, to a linear algebra question over the binary field $\F=\{0,1\}$. All matrices for the remainder of the paper are assumed to be over $\F$ unless stated otherwise.

\subsection{The cell-adjacency graph}

\begin{definition}
\label{def:cellgraph}
The \emph{cell-adjacency graph} $G_\lambda$ has a vertex for each box in $[\lambda]$.
Two boxes are adjacent when they share a side:
\[
(i,j)\sim(k,l)
\quad\Longleftrightarrow\quad
|i-k|+|j-l|=1.
\]
We assume the planar embedding obtained by placing each vertex at the
center of its box and drawing each edge as a horizontal or vertical line segment crossing the
shared side.
\end{definition}
\begin{remark}
\label{remark:medialdefin}
Note that the medial graph we define below depends not just on the isomorphism type of a planar graph, but on the particular embedding of the planar graph, which is why we specify the particular drawing of $G_\lambda$.
\end{remark}

The graph $G_\lambda$ is connected, planar, bipartite, and has maximum
degree four.
It has
$
|V(G_\lambda)|=|\lambda|.
$
Adding the number of horizontal and vertical edges gives:
\begin{equation}
\label{eq:edgecount}
|E(G_\lambda)|
=
\sum_{i=1}^{\ell}(\lambda_i-1)
+
\sum_{i=2}^{\ell}\lambda_i
=
2|\lambda|-\lambda_1-\ell.
\end{equation}
Since $G_\lambda$ is connected, its cycle-space (defined below in Definition \ref{def:cyclescocyclesbicycles}) dimension is \cite[Theorem 14.2.1]{GodsilRoyle}:
\begin{equation}
\label{eq:cycledim}
|E(G_\lambda)|-|V(G_\lambda)|+1
=
|\lambda|-\lambda_1-\ell+1
=
\sum_{i=2}^{\ell}(\lambda_i-1).
\end{equation}

\subsection{Medial graphs and medial links}

The connection between the billiard curves and graph theory is provided by
the \emph{medial graph}:

\begin{definition} \cite[17.2]{GodsilRoyle}
\label{def:medialgraph}
Let $G$ be a planar graph together with a fixed
embedding in the plane. The \emph{medial graph} $\mathcal M(G)$ has one
vertex $v_e$ for each edge $e$ of $G$. For each face $F$ of $G$, read the edges
$
e_1,e_2,\ldots,e_r
$
in order around the boundary of $F$. For each $i$, join
$v_{e_i}$ to $v_{e_{i+1}}$, with indices taken modulo $r$.

The unbounded face is included (or alternately one can think of $G$ drawn on the surface of a sphere so all faces are bounded). If $G$ has a face with an edge $e$ adjacent to a vertex of degree one (a leaf), the face containing that edge is viewed as having two consecutive occurrences of $e$ and hence gives a loop in $\mathcal{M}(G)$. Similarly a loop $e$ in $G$ bounds a face and contributes a loop from $v_e$ to $v_e$ in $\mathcal{M}(G)$.
\end{definition}

The medial graph is $4$-regular, where loops and multiple edges are allowed.
As noted in Remark \ref{remark:medialdefin}, $\mathcal M(G)$ may depend on the chosen plane
embedding of $G$, not merely on the abstract isomorphism type of $G$.

Figure~\ref{fig:graph-and-medial} gives an example of a plane graph $G$
with $13$ vertices and $18$ edges and its medial graph $\mathcal M(G)$.
The twelve curved  edges around the outside of $\mathcal{M}(G)$ arise from the unbounded face, with the leaf traversed twice, producing the loop in the upper right of the drawing.

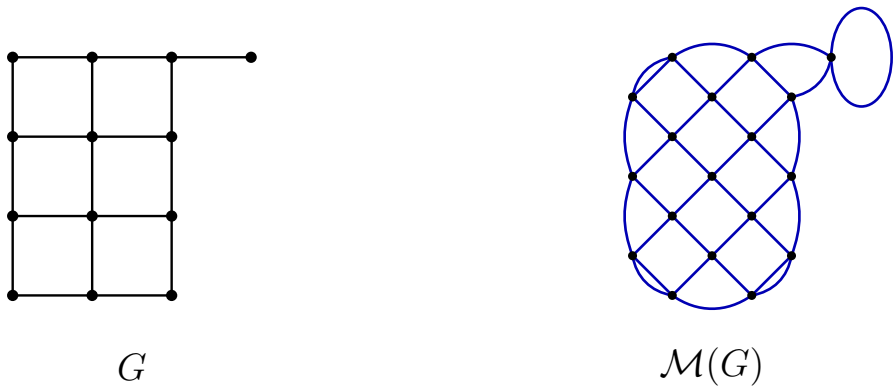
\begin{figure}[ht]
\centering
\begin{tikzpicture}[
    scale=1.05,
    graphv/.style={circle,fill=black,inner sep=1.5pt},
    medv/.style={circle,fill=black,inner sep=1.2pt},
    graph/.style={line width=0.9pt},
    med/.style={blue!70!black, line width=1pt},
    outermed/.style={blue!70!black, line width=1pt}
]

\begin{scope}[xshift=0cm]
  \coordinate (a1) at (0,3);
  \coordinate (a2) at (1,3);
  \coordinate (a3) at (2,3);
  \coordinate (a4) at (3,3);

  \coordinate (b1) at (0,2);
  \coordinate (b2) at (1,2);
  \coordinate (b3) at (2,2);

  \coordinate (c1) at (0,1);
  \coordinate (c2) at (1,1);
  \coordinate (c3) at (2,1);

  \coordinate (d1) at (0,0);
  \coordinate (d2) at (1,0);
  \coordinate (d3) at (2,0);

  \draw[graph] (a1)--(a2)--(a3)--(a4);
  \draw[graph] (b1)--(b2)--(b3);
  \draw[graph] (c1)--(c2)--(c3);
  \draw[graph] (d1)--(d2)--(d3);

  \draw[graph] (a1)--(b1)--(c1)--(d1);
  \draw[graph] (a2)--(b2)--(c2)--(d2);
  \draw[graph] (a3)--(b3)--(c3)--(d3);

  \foreach \p in {a1,a2,a3,a4,b1,b2,b3,c1,c2,c3,d1,d2,d3}
    \node[graphv] at (\p) {};

  \node at (1.5,-0.9) {\Large $G$};
\end{scope}

\begin{scope}[xshift=7.8cm]
  \coordinate (t12) at (0.5,3);
  \coordinate (t23) at (1.5,3);
  \coordinate (t34) at (2.5,3);

  \coordinate (u12) at (0.5,2);
  \coordinate (u23) at (1.5,2);

  \coordinate (m12) at (0.5,1);
  \coordinate (m23) at (1.5,1);

  \coordinate (b12m) at (0.5,0);
  \coordinate (b23m) at (1.5,0);

  \coordinate (l11) at (0,2.5);
  \coordinate (l22) at (1,2.5);
  \coordinate (l33) at (2,2.5);

  \coordinate (n11) at (0,1.5);
  \coordinate (n22) at (1,1.5);
  \coordinate (n33) at (2,1.5);

  \coordinate (p11) at (0,0.5);
  \coordinate (p22) at (1,0.5);
  \coordinate (p33) at (2,0.5);

  \draw[med] (t12)--(l22)--(u12)--(l11)--cycle;
  \draw[med] (t23)--(l33)--(u23)--(l22)--cycle;
  \draw[med] (u12)--(n22)--(m12)--(n11)--cycle;
  \draw[med] (u23)--(n33)--(m23)--(n22)--cycle;
  \draw[med] (m12)--(p22)--(b12m)--(p11)--cycle;
  \draw[med] (m23)--(p33)--(b23m)--(p22)--cycle;

  \draw[outermed] (l11) to[bend left=35] (t12);
  \draw[outermed] (t12) to[bend left=35] (t23);
  \draw[outermed] (t23) to[bend left=35] (t34);

  \draw[outermed]
    ($(t34)+(0.38,0)$)
    circle [x radius=0.38, y radius=0.62];

  \draw[outermed] (t34) to[bend left=35] (l33);
  \draw[outermed] (l33) to[bend left=20] (n33);
  \draw[outermed] (n33) to[bend left=20] (p33);
  \draw[outermed] (p33) to[bend left=35] (b23m);
  \draw[outermed] (b23m) to[bend left=35] (b12m);
  \draw[outermed] (b12m) to[bend left=35] (p11);
  \draw[outermed] (p11) to[bend left=20] (n11);
  \draw[outermed] (n11) to[bend left=20] (l11);

  \foreach \p in {t12,t23,t34,u12,u23,m12,m23,b12m,b23m,
                  l11,l22,l33,n11,n22,n33,p11,p22,p33}
    \node[medv] at (\p) {};

  \node at (1.0,-0.9) {\Large $\mathcal M(G)$};
\end{scope}

\end{tikzpicture}
\caption{A plane graph $G$ and its medial graph $\mathcal M(G)$.}
\label{fig:graph-and-medial}
\end{figure}

There is an equivalent topological construction of $\mathcal{M}(G)$. One can take the boundary of a small regular neighborhood around $G$ and then pinch each edge in the center to create a vertex of degree 4, see \cite{SilverWilliams}.

Since $\mathcal{M} (G)$ is $4$-regular, we may regard each of its vertices
as a crossing. Resolving each crossing in one of two ways so one arc passes over the other produces a
diagram of a link. Different choices may produce different links, but do not affect the number of link components, which is all we are interested in.
We call the resulting collection of closed curves the \emph{medial link}
of $G$. The key observation is that the billiard curves in the Young diagram of $\lambda$ are precisely the components of the medial link of $G_\lambda$.

\begin{proposition}
\label{prop:medialidentification}
Suppose $|\lambda|>1$. The diagonal billiard curves in $[\lambda]$ are
isotopic, as an immersed collection of closed curves in the plane, to
the medial link of $G_\lambda$. In particular, there is a natural
bijection between billiard orbits and medial-link components.
\end{proposition}

\begin{proof}
Figure~\ref{fig:globalmedial} illustrates the entire construction for
$\lambda=(5,4,4,4)$, the map we describe applies equally well
to any Young diagram.

Begin with the billiard curves. At each reflection off the exterior
boundary, round the corner and suppress the reflection point, as in the
passage from Figure~\ref{fig:globalmedial}(a) to (b). This preserves the number of
components.

Now examine the rounded picture locally near each cell in the Young diagram.  The center of
a cell is a vertex $v$ of $G_\lambda$, and the edges incident with
$v$ cross precisely those sides of the cell that are shared with
neighboring cells.  Inside the cell, the rounded billiard strands join
the corresponding edge-midpoints in consecutive pairs in the cyclic
order around $v$.  These are exactly the local medial edges drawn
around the vertex $v$.

These local deformations may be performed simultaneously in the
interiors of the cells, while fixing the midpoints of the shared sides.
The rounded arcs running along the exterior boundary are then deformed
in a collar of that boundary onto the medial edges belonging to the
unbounded face.  The resulting collection of curves is the
medial graph $\mathcal M(G_\lambda)$, as shown in
Figure~\ref{fig:globalmedial}(c).

Finally, a side shared by two boxes corresponds to an edge $e$ of
$G_\lambda$, and its midpoint is the corresponding vertex of the medial
graph. The billiard trajectory passes straight through the shared side,
so the four incident half-arcs are paired as opposite half-edges. This
is precisely the pairing used when the $4$-valent medial vertex is
regarded as a link crossing.

Therefore the deformation shown in Figure~\ref{fig:globalmedial}
carries the billiard orbits bijectively to the medial-link components, as desired.

\end{proof}

\begin{figure}[ht]
\centering
\begin{subfigure}[t]{0.29\textwidth}
\centering
\begin{tikzpicture}[x=.72cm,y=.72cm]
\draw[black!75,line width=0.35pt] (0,3) rectangle (1,4);
\draw[black!75,line width=0.35pt] (1,3) rectangle (2,4);
\draw[black!75,line width=0.35pt] (2,3) rectangle (3,4);
\draw[black!75,line width=0.35pt] (3,3) rectangle (4,4);
\draw[black!75,line width=0.35pt] (4,3) rectangle (5,4);
\draw[black!75,line width=0.35pt] (0,2) rectangle (1,3);
\draw[black!75,line width=0.35pt] (1,2) rectangle (2,3);
\draw[black!75,line width=0.35pt] (2,2) rectangle (3,3);
\draw[black!75,line width=0.35pt] (3,2) rectangle (4,3);
\draw[black!75,line width=0.35pt] (0,1) rectangle (1,2);
\draw[black!75,line width=0.35pt] (1,1) rectangle (2,2);
\draw[black!75,line width=0.35pt] (2,1) rectangle (3,2);
\draw[black!75,line width=0.35pt] (3,1) rectangle (4,2);
\draw[black!75,line width=0.35pt] (0,0) rectangle (1,1);
\draw[black!75,line width=0.35pt] (1,0) rectangle (2,1);
\draw[black!75,line width=0.35pt] (2,0) rectangle (3,1);
\draw[black!75,line width=0.35pt] (3,0) rectangle (4,1);

\draw[orbitblue,line width=1.55pt,line cap=round,line join=round]
(0.500,4.000) -- (1.000,3.500) -- (1.500,3.000) --
(2.000,2.500) -- (2.500,2.000) -- (3.000,1.500) --
(3.500,1.000) -- (4.000,0.500) -- (3.500,0.000) --
(3.000,0.500) -- (2.500,1.000) -- (2.000,1.500) --
(1.500,2.000) -- (1.000,2.500) -- (0.500,3.000) --
(0.000,3.500) -- (0.500,4.000);

\draw[orbitred,line width=1.55pt,line cap=round,line join=round]
(1.000,3.500) -- (0.500,3.000) -- (0.000,2.500) --
(0.500,2.000) -- (1.000,1.500) -- (1.500,1.000) --
(2.000,0.500) -- (2.500,0.000) -- (3.000,0.500) --
(3.500,1.000) -- (4.000,1.500) -- (3.500,2.000) --
(3.000,2.500) -- (2.500,3.000) -- (2.000,3.500) --
(1.500,4.000) -- (1.000,3.500);

\draw[orbitgreen,line width=1.55pt,line cap=round,line join=round]
(2.000,3.500) -- (1.500,3.000) -- (1.000,2.500) --
(0.500,2.000) -- (0.000,1.500) -- (0.500,1.000) --
(1.000,0.500) -- (1.500,0.000) -- (2.000,0.500) --
(2.500,1.000) -- (3.000,1.500) -- (3.500,2.000) --
(4.000,2.500) -- (3.500,3.000) -- (3.000,3.500) --
(2.500,4.000) -- (2.000,3.500);

\draw[orbitorange,line width=1.55pt,line cap=round,line join=round]
(3.000,3.500) -- (2.500,3.000) -- (2.000,2.500) --
(1.500,2.000) -- (1.000,1.500) -- (0.500,1.000) --
(0.000,0.500) -- (0.500,0.000) -- (1.000,0.500) --
(1.500,1.000) -- (2.000,1.500) -- (2.500,2.000) --
(3.000,2.500) -- (3.500,3.000) -- (4.000,3.500) --
(4.500,4.000) -- (5.000,3.500) -- (4.500,3.000) --
(4.000,3.500) -- (3.500,4.000) -- (3.000,3.500);
\end{tikzpicture}
\caption{The original billiard orbits.}
\end{subfigure}%
\hfill
\raisebox{1.65cm}{$\Longrightarrow$}%
\hfill
\begin{subfigure}[t]{0.29\textwidth}
\centering
\begin{tikzpicture}[x=.72cm,y=.72cm]
\draw[black!75,line width=0.35pt] (0,3) rectangle (1,4);
\draw[black!75,line width=0.35pt] (1,3) rectangle (2,4);
\draw[black!75,line width=0.35pt] (2,3) rectangle (3,4);
\draw[black!75,line width=0.35pt] (3,3) rectangle (4,4);
\draw[black!75,line width=0.35pt] (4,3) rectangle (5,4);
\draw[black!75,line width=0.35pt] (0,2) rectangle (1,3);
\draw[black!75,line width=0.35pt] (1,2) rectangle (2,3);
\draw[black!75,line width=0.35pt] (2,2) rectangle (3,3);
\draw[black!75,line width=0.35pt] (3,2) rectangle (4,3);
\draw[black!75,line width=0.35pt] (0,1) rectangle (1,2);
\draw[black!75,line width=0.35pt] (1,1) rectangle (2,2);
\draw[black!75,line width=0.35pt] (2,1) rectangle (3,2);
\draw[black!75,line width=0.35pt] (3,1) rectangle (4,2);
\draw[black!75,line width=0.35pt] (0,0) rectangle (1,1);
\draw[black!75,line width=0.35pt] (1,0) rectangle (2,1);
\draw[black!75,line width=0.35pt] (2,0) rectangle (3,1);
\draw[black!75,line width=0.35pt] (3,0) rectangle (4,1);


\draw[orbitblue,line width=1.55pt,line cap=round,line join=round,
rounded corners=5.5pt]
(0.750,3.750) --
(1.000,3.500) -- (1.500,3.000) --
(2.000,2.500) -- (2.500,2.000) -- (3.000,1.500) --
(3.500,1.000) -- (4.000,0.500) -- (3.500,0.000) --
(3.000,0.500) -- (2.500,1.000) -- (2.000,1.500) --
(1.500,2.000) -- (1.000,2.500) -- (0.500,3.000) --
(0.000,3.500) -- (0.500,4.000) --
(0.750,3.750);

\draw[orbitred,line width=1.55pt,line cap=round,line join=round,
rounded corners=5.5pt]
(0.750,3.250) --
(0.500,3.000) -- (0.000,2.500) --
(0.500,2.000) -- (1.000,1.500) -- (1.500,1.000) --
(2.000,0.500) -- (2.500,0.000) -- (3.000,0.500) --
(3.500,1.000) -- (4.000,1.500) -- (3.500,2.000) --
(3.000,2.500) -- (2.500,3.000) -- (2.000,3.500) --
(1.500,4.000) -- (1.000,3.500) --
(0.750,3.250);

\draw[orbitgreen,line width=1.55pt,line cap=round,line join=round,
rounded corners=5.5pt]
(1.750,3.250) --
(1.500,3.000) -- (1.000,2.500) --
(0.500,2.000) -- (0.000,1.500) -- (0.500,1.000) --
(1.000,0.500) -- (1.500,0.000) -- (2.000,0.500) --
(2.500,1.000) -- (3.000,1.500) -- (3.500,2.000) --
(4.000,2.500) -- (3.500,3.000) -- (3.000,3.500) --
(2.500,4.000) -- (2.000,3.500) --
(1.750,3.250);

\draw[orbitorange,line width=1.55pt,line cap=round,line join=round,
rounded corners=5.5pt]
(2.750,3.250) --
(2.500,3.000) -- (2.000,2.500) --
(1.500,2.000) -- (1.000,1.500) -- (0.500,1.000) --
(0.000,0.500) -- (0.500,0.000) -- (1.000,0.500) --
(1.500,1.000) -- (2.000,1.500) -- (2.500,2.000) --
(3.000,2.500) -- (3.500,3.000) -- (4.000,3.500) --
(4.500,4.000) -- (5.000,3.500) -- (4.500,3.000) --
(4.000,3.500) -- (3.500,4.000) -- (3.000,3.500) --
(2.750,3.250);
\end{tikzpicture}
\caption{Round each boundary ricochet.}
\end{subfigure}%
\hfill
\raisebox{1.65cm}{$\Longrightarrow$}%
\hfill
\begin{subfigure}[t]{0.29\textwidth}
\centering
\begin{tikzpicture}[x=.72cm,y=.72cm]

\draw[gray!65,line width=0.65pt] (0.500,3.500) -- (1.500,3.500);
\draw[gray!65,line width=0.65pt] (0.500,3.500) -- (0.500,2.500);
\draw[gray!65,line width=0.65pt] (1.500,3.500) -- (2.500,3.500);
\draw[gray!65,line width=0.65pt] (1.500,3.500) -- (1.500,2.500);
\draw[gray!65,line width=0.65pt] (2.500,3.500) -- (3.500,3.500);
\draw[gray!65,line width=0.65pt] (2.500,3.500) -- (2.500,2.500);
\draw[gray!65,line width=0.65pt] (3.500,3.500) -- (4.500,3.500);
\draw[gray!65,line width=0.65pt] (3.500,3.500) -- (3.500,2.500);

\draw[gray!65,line width=0.65pt] (0.500,2.500) -- (1.500,2.500);
\draw[gray!65,line width=0.65pt] (0.500,2.500) -- (0.500,1.500);
\draw[gray!65,line width=0.65pt] (1.500,2.500) -- (2.500,2.500);
\draw[gray!65,line width=0.65pt] (1.500,2.500) -- (1.500,1.500);
\draw[gray!65,line width=0.65pt] (2.500,2.500) -- (3.500,2.500);
\draw[gray!65,line width=0.65pt] (2.500,2.500) -- (2.500,1.500);
\draw[gray!65,line width=0.65pt] (3.500,2.500) -- (3.500,1.500);

\draw[gray!65,line width=0.65pt] (0.500,1.500) -- (1.500,1.500);
\draw[gray!65,line width=0.65pt] (0.500,1.500) -- (0.500,0.500);
\draw[gray!65,line width=0.65pt] (1.500,1.500) -- (2.500,1.500);
\draw[gray!65,line width=0.65pt] (1.500,1.500) -- (1.500,0.500);
\draw[gray!65,line width=0.65pt] (2.500,1.500) -- (3.500,1.500);
\draw[gray!65,line width=0.65pt] (2.500,1.500) -- (2.500,0.500);
\draw[gray!65,line width=0.65pt] (3.500,1.500) -- (3.500,0.500);

\draw[gray!65,line width=0.65pt] (0.500,0.500) -- (1.500,0.500);
\draw[gray!65,line width=0.65pt] (1.500,0.500) -- (2.500,0.500);
\draw[gray!65,line width=0.65pt] (2.500,0.500) -- (3.500,0.500);

\foreach \p in {
(0.500,3.500),(1.500,3.500),(2.500,3.500),(3.500,3.500),(4.500,3.500),
(0.500,2.500),(1.500,2.500),(2.500,2.500),(3.500,2.500),
(0.500,1.500),(1.500,1.500),(2.500,1.500),(3.500,1.500),
(0.500,0.500),(1.500,0.500),(2.500,0.500),(3.500,0.500)}
\fill[gray!70] \p circle (1.4pt);

\draw[orbitred,line width=1.30pt,line cap=round]
plot[smooth,tension=0.75] coordinates
{(1.000,3.500) (0.760,3.240) (0.500,3.000)};
\draw[orbitblue,line width=1.30pt,line cap=round]
plot[smooth,tension=0.75] coordinates
{(0.500,3.000) (0.240,3.240) (0.240,3.760)
 (0.760,3.760) (1.000,3.500)};

\draw[orbitgreen,line width=1.30pt,line cap=round]
plot[smooth,tension=0.75] coordinates
{(2.000,3.500) (1.760,3.240) (1.500,3.000)};
\draw[orbitblue,line width=1.30pt,line cap=round]
plot[smooth,tension=0.75] coordinates
{(1.500,3.000) (1.240,3.240) (1.000,3.500)};
\draw[orbitred,line width=1.30pt,line cap=round]
plot[smooth,tension=0.75] coordinates
{(1.000,3.500) (1.240,3.760) (1.760,3.760)
 (2.000,3.500)};

\draw[orbitorange,line width=1.30pt,line cap=round]
plot[smooth,tension=0.75] coordinates
{(3.000,3.500) (2.760,3.240) (2.500,3.000)};
\draw[orbitred,line width=1.30pt,line cap=round]
plot[smooth,tension=0.75] coordinates
{(2.500,3.000) (2.240,3.240) (2.000,3.500)};
\draw[orbitgreen,line width=1.30pt,line cap=round]
plot[smooth,tension=0.75] coordinates
{(2.000,3.500) (2.240,3.760) (2.760,3.760)
 (3.000,3.500)};

\draw[orbitorange,line width=1.30pt,line cap=round]
plot[smooth,tension=0.75] coordinates
{(4.000,3.500) (3.760,3.240) (3.500,3.000)};
\draw[orbitgreen,line width=1.30pt,line cap=round]
plot[smooth,tension=0.75] coordinates
{(3.500,3.000) (3.240,3.240) (3.000,3.500)};
\draw[orbitorange,line width=1.30pt,line cap=round]
plot[smooth,tension=0.75] coordinates
{(3.000,3.500) (3.240,3.760) (3.760,3.760)
 (4.000,3.500)};
\draw[orbitorange,line width=1.30pt,line cap=round]
plot[smooth cycle,tension=0.75] coordinates
{(4.000,3.500) (4.240,3.760) (4.760,3.760)
 (4.760,3.240) (4.240,3.240)};

\draw[orbitblue,line width=1.30pt,line cap=round]
plot[smooth,tension=0.75] coordinates
{(0.500,3.000) (0.760,2.760) (1.000,2.500)};
\draw[orbitgreen,line width=1.30pt,line cap=round]
plot[smooth,tension=0.75] coordinates
{(1.000,2.500) (0.760,2.240) (0.500,2.000)};
\draw[orbitred,line width=1.30pt,line cap=round]
plot[smooth,tension=0.75] coordinates
{(0.500,2.000) (0.240,2.240) (0.240,2.760)
 (0.500,3.000)};

\draw[orbitblue,line width=1.30pt,line cap=round]
plot[smooth,tension=0.75] coordinates
{(1.500,3.000) (1.760,2.760) (2.000,2.500)};
\draw[orbitorange,line width=1.30pt,line cap=round]
plot[smooth,tension=0.75] coordinates
{(2.000,2.500) (1.760,2.240) (1.500,2.000)};
\draw[orbitblue,line width=1.30pt,line cap=round]
plot[smooth,tension=0.75] coordinates
{(1.500,2.000) (1.240,2.240) (1.000,2.500)};
\draw[orbitgreen,line width=1.30pt,line cap=round]
plot[smooth,tension=0.75] coordinates
{(1.000,2.500) (1.240,2.760) (1.500,3.000)};

\draw[orbitred,line width=1.30pt,line cap=round]
plot[smooth,tension=0.75] coordinates
{(2.500,3.000) (2.760,2.760) (3.000,2.500)};
\draw[orbitorange,line width=1.30pt,line cap=round]
plot[smooth,tension=0.75] coordinates
{(3.000,2.500) (2.760,2.240) (2.500,2.000)};
\draw[orbitblue,line width=1.30pt,line cap=round]
plot[smooth,tension=0.75] coordinates
{(2.500,2.000) (2.240,2.240) (2.000,2.500)};
\draw[orbitorange,line width=1.30pt,line cap=round]
plot[smooth,tension=0.75] coordinates
{(2.000,2.500) (2.240,2.760) (2.500,3.000)};

\draw[orbitgreen,line width=1.30pt,line cap=round]
plot[smooth,tension=0.75] coordinates
{(3.500,3.000) (3.760,2.760) (3.760,2.240)
 (3.500,2.000)};
\draw[orbitred,line width=1.30pt,line cap=round]
plot[smooth,tension=0.75] coordinates
{(3.500,2.000) (3.240,2.240) (3.000,2.500)};
\draw[orbitorange,line width=1.30pt,line cap=round]
plot[smooth,tension=0.75] coordinates
{(3.000,2.500) (3.240,2.760) (3.500,3.000)};

\draw[orbitred,line width=1.30pt,line cap=round]
plot[smooth,tension=0.75] coordinates
{(0.500,2.000) (0.760,1.760) (1.000,1.500)};
\draw[orbitorange,line width=1.30pt,line cap=round]
plot[smooth,tension=0.75] coordinates
{(1.000,1.500) (0.760,1.240) (0.500,1.000)};
\draw[orbitgreen,line width=1.30pt,line cap=round]
plot[smooth,tension=0.75] coordinates
{(0.500,1.000) (0.240,1.240) (0.240,1.760)
 (0.500,2.000)};

\draw[orbitblue,line width=1.30pt,line cap=round]
plot[smooth,tension=0.75] coordinates
{(1.500,2.000) (1.760,1.760) (2.000,1.500)};
\draw[orbitorange,line width=1.30pt,line cap=round]
plot[smooth,tension=0.75] coordinates
{(2.000,1.500) (1.760,1.240) (1.500,1.000)};
\draw[orbitred,line width=1.30pt,line cap=round]
plot[smooth,tension=0.75] coordinates
{(1.500,1.000) (1.240,1.240) (1.000,1.500)};
\draw[orbitorange,line width=1.30pt,line cap=round]
plot[smooth,tension=0.75] coordinates
{(1.000,1.500) (1.240,1.760) (1.500,2.000)};

\draw[orbitblue,line width=1.30pt,line cap=round]
plot[smooth,tension=0.75] coordinates
{(2.500,2.000) (2.760,1.760) (3.000,1.500)};
\draw[orbitgreen,line width=1.30pt,line cap=round]
plot[smooth,tension=0.75] coordinates
{(3.000,1.500) (2.760,1.240) (2.500,1.000)};
\draw[orbitblue,line width=1.30pt,line cap=round]
plot[smooth,tension=0.75] coordinates
{(2.500,1.000) (2.240,1.240) (2.000,1.500)};
\draw[orbitorange,line width=1.30pt,line cap=round]
plot[smooth,tension=0.75] coordinates
{(2.000,1.500) (2.240,1.760) (2.500,2.000)};

\draw[orbitred,line width=1.30pt,line cap=round]
plot[smooth,tension=0.75] coordinates
{(3.500,2.000) (3.760,1.760) (3.760,1.240)
 (3.500,1.000)};
\draw[orbitblue,line width=1.30pt,line cap=round]
plot[smooth,tension=0.75] coordinates
{(3.500,1.000) (3.240,1.240) (3.000,1.500)};
\draw[orbitgreen,line width=1.30pt,line cap=round]
plot[smooth,tension=0.75] coordinates
{(3.000,1.500) (3.240,1.760) (3.500,2.000)};

\draw[orbitgreen,line width=1.30pt,line cap=round]
plot[smooth,tension=0.75] coordinates
{(0.500,1.000) (0.760,0.760) (1.000,0.500)};
\draw[orbitorange,line width=1.30pt,line cap=round]
plot[smooth,tension=0.75] coordinates
{(1.000,0.500) (0.760,0.240) (0.240,0.240)
 (0.240,0.760) (0.500,1.000)};

\draw[orbitred,line width=1.30pt,line cap=round]
plot[smooth,tension=0.75] coordinates
{(1.500,1.000) (1.760,0.760) (2.000,0.500)};
\draw[orbitgreen,line width=1.30pt,line cap=round]
plot[smooth,tension=0.75] coordinates
{(2.000,0.500) (1.760,0.240) (1.240,0.240)
 (1.000,0.500)};
\draw[orbitorange,line width=1.30pt,line cap=round]
plot[smooth,tension=0.75] coordinates
{(1.000,0.500) (1.240,0.760) (1.500,1.000)};

\draw[orbitblue,line width=1.30pt,line cap=round]
plot[smooth,tension=0.75] coordinates
{(2.500,1.000) (2.760,0.760) (3.000,0.500)};
\draw[orbitred,line width=1.30pt,line cap=round]
plot[smooth,tension=0.75] coordinates
{(3.000,0.500) (2.760,0.240) (2.240,0.240)
 (2.000,0.500)};
\draw[orbitgreen,line width=1.30pt,line cap=round]
plot[smooth,tension=0.75] coordinates
{(2.000,0.500) (2.240,0.760) (2.500,1.000)};

\draw[orbitblue,line width=1.30pt,line cap=round]
plot[smooth,tension=0.75] coordinates
{(3.500,1.000) (3.760,0.760) (3.760,0.240)
 (3.240,0.240) (3.000,0.500)};
\draw[orbitred,line width=1.30pt,line cap=round]
plot[smooth,tension=0.75] coordinates
{(3.000,0.500) (3.240,0.760) (3.500,1.000)};

\end{tikzpicture}
\caption{Intersection points are vertices of the medial graph of $G_\lambda$.}
\end{subfigure}

\caption{The bijection for $\lambda=(5,4,4,4)$.
Panel (a) is the four-orbit billiard picture from
Figure~\ref{fig:orbitexamples}. In (b), segments meeting at reflection points are
rounded into a single smooth arc, so no component is created or
destroyed. The gray vertices and edges in (c) are the
cell-adjacency graph $G_\lambda$. At the midpoint of each gray edge the
colored strands meet in the same opposite-half-edge pairing as the
billiard strands crossing the corresponding shared side in the original Young diagram.}
\label{fig:globalmedial}
\end{figure}
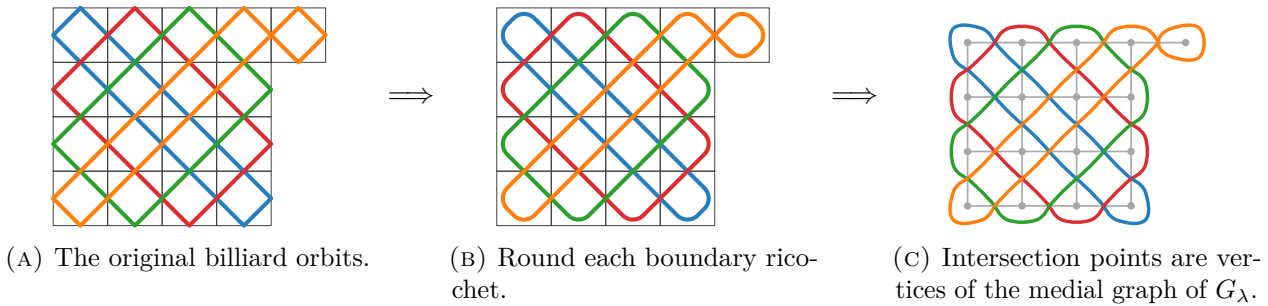

\subsection{Cycle space, cut space, and bicycle space}
Next we transform the problem of counting components of the medial link into linear algebra. We proceed in the standard way, constructing a vector space over the field $\F$ with a basis given by the edges $E(G)$. So let $G=(V,E)$ be a finite graph. Identify a subset of edges with its
characteristic vector in $\F^E$; under this identification, vector
addition corresponds to symmetric difference on the level of subsets. For further background, see
\cite{GodsilRoyle}.

\begin{definition}
\label{def:cyclescocyclesbicycles}
The \emph{cycle space} $\cC(G)$ (sometimes called the \emph{flow space}) is the subspace of $\F^E$ spanned by
the characteristic vectors of the cycles of $G$. Equivalently,
$\cC(G)$ consists precisely of subsets of edges so that
every vertex is incident to an even number of them.

The \emph{cut space}, or \emph{cocycle space}, is the orthogonal complement
\[
\mathcal C^*(G)=\cC(G)^\perp.
\]
Thus a subset of edges lies in the cut space if and only if it meets every
element of the cycle space in an even number of edges.
Equivalently, $\mathcal C^*(G)$ is the subspace of $\F^E$ spanned by
\emph{edge cuts}
\[
\delta(S):=\{uv\in E:u\in S,\ v\notin S\},
\qquad S\subseteq V.
\]
\end{definition}

\begin{definition}
\label{def:bicyclespace}

The \emph{bicycle space} is
\[
\cB(G)=\cC(G)\cap\mathcal C^*(G).
\]
Thus a \emph{bicycle} is an edge set which has even degree at every vertex and
which meets every cycle of $G$ in an even number of edges.
\end{definition}

\begin{example}
\label{ex:k4bicycle}
Consider the complete graph $G=K_4$ on vertices
\[
V(G)=\{v_1,v_2,v_3,v_4\},
\]
and let
\[
B=\{v_1v_2,\ v_2v_3,\ v_3v_4,\ v_4v_1\}.
\]
so $B$ is a four-cycle and thus $B\in\cC(G)$. If we choose
$
S=\{v_1,v_3\},
$
then
$
B=\delta(S)
$
and so $B\in\cC^*(G)$, i.e. $B$ is also an edge cut. Therefore $B\in\cB(G)$ is a nonzero bicycle. Figure \ref{fig:k4bicycle} illustrates this example. One can check that any cycle in $G$ meets $B$ in an even number of edges.
\end{example}

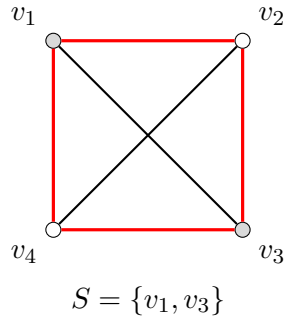
\begin{figure}[ht]
\centering
\begin{tikzpicture}[
    scale=1.25,
    vertex/.style={circle,draw,fill=white,inner sep=2pt},
    setvertex/.style={circle,draw,fill=gray!30,inner sep=2pt}
]

\node[setvertex,label=above left:$v_1$]  (v1) at (0,2) {};
\node[vertex,label=above right:$v_2$]    (v2) at (2,2) {};
\node[setvertex,label=below right:$v_3$] (v3) at (2,0) {};
\node[vertex,label=below left:$v_4$]     (v4) at (0,0) {};

\draw[very thick,red] (v1)--(v2);
\draw[very thick,red] (v2)--(v3);
\draw[very thick,red] (v3)--(v4);
\draw[very thick,red] (v4)--(v1);

\draw[thick] (v1)--(v3);
\draw[thick] (v2)--(v4);

\node[align=center] at (1,-0.75)
{$S=\{v_1,v_3\}$};

\end{tikzpicture}
\caption{A nonzero bicycle in $K_4$. The four highlighted edges form
the cycle $v_1v_2v_3v_4v_1$. They also form the edge cut
$\delta(\{v_1,v_3\})$.}
\label{fig:k4bicycle}
\end{figure}

\subsection{Medial link components and bicycles}

For a component $K$ of the medial link $\mathcal{M}(G)$, define its \emph{residue} $\rho(K)\in\F^{E(G)}$ by declaring that the coefficient of an edge $e\in E(G)$ is the number, modulo two, of times that $K$ passes through the corresponding medial vertex $m_e$. Equivalently, $\rho(K)$ is the subset of edges $e$ for which $K$ passes through $m_e$ an odd number of times. This residue lies in both the cycle space and the cut space of $G$, and hence
$ \rho(K)\in\cB(G).$

The following  theorem translates our orbit count into linear algebra.

\begin{theorem}
\cite[Theorem~17.3.5]{GodsilRoyle}
\label{thm:medialbicycle}
Let $G$ be a connected plane graph whose medial link has $c$ components.
The residues of any choice of $c-1$ components form a basis of
$\cB(G)$. Consequently,
\[
c=1+\dim_{\F}\cB(G).
\]
\end{theorem}

An elementary proof in the language
of medial links is given by Lamey--Silver--Williams
\cite[Theorem~4.4]{LameySilverWilliams}. That the $c$
residues are linearly dependent is easy to see: every graph edge is crossed by the medial link
twice in total, so the sum of all residues is zero (working over $\F$). The substantive
assertion is that there are no further linear relations.

Combining Proposition~\ref{prop:medialidentification} and
Theorem~\ref{thm:medialbicycle}, and checking $\lambda=(1)$ directly,
gives the following:

\begin{corollary}
\label{cor:sigmabicycle}
For every nonempty partition $\lambda$,
\[
\sigma(\lambda)=1+\dim_{\F}\cB(G_\lambda).
\]
\end{corollary}

\subsection{The incidence matrix and the mod-$2$ Laplacian}

Corollary \ref{cor:sigmabicycle} has reduced the problem of counting orbits to determining the dimension of the bicycle space $\cB(G_\lambda)$. We now make this even more explicit in terms of a single matrix. The cycle and cut spaces have natural descriptions in terms of the
incidence matrix of $G$.

\begin{definition}
The \emph{incidence matrix} of a finite simple graph $G$ is
the matrix $N=N(G)$
whose rows are indexed by vertices, columns are indexed by edges,
and whose entries are
\[
N_{v,e}
=
\begin{cases}
1,&\text{if $v$ is an endpoint of $e$,}\\
0,&\text{otherwise.}
\end{cases}
\]
\end{definition}

Since $G$ is simple, every column of $N$ contains exactly two nonzero entries. If
a vector $x\in\F^E$ is regarded as a subset of the edges, then the $v$-coordinate of
$Nx$ is the degree of $v$ in the edge set $x$, taken modulo two.
Consequently,
\[
\cC(G)=\ker N.
\]

The row of $N$ corresponding to a vertex $v$ includes all the edges adjacent to $v$, and thus is the characteristic
vector of the cut $\delta(\{v\})$. More generally, summing the rows
indexed by a subset $S\subseteq V$ gives the characteristic vector of
$\delta(S)$, as edges with both endpoints in $S$ occur twice
and hence cancel over $\F$. Therefore
\[
\mathcal C^*(G)
=
\operatorname{row}(N)
=
\operatorname{im}N^{\mathsf T},
\]
and
\[
\cB(G)
=
\ker N\cap\operatorname{row}(N).
\]

Let $A(G)$ denote the adjacency matrix of $G$, and let $D(G)$ be the
diagonal matrix with the degree of the vertices down the diagonal. Over $\F$ the mod-$2$ Laplacian matrix is
\[
L(G)=D(G)+A(G).
\]
Since every column of $N$ records the two endpoints of the corresponding
edge,
\[
NN^{\mathsf T}=L(G).
\]
The dimension of the bicycle space is determined by the nullity of $L(G)$:
\begin{proposition}
\label{prop:lapbicycle}
If $G$ is connected, then
\[
\dim_{\F}\cB(G)=\nul_{\F}L(G)-1.
\]
More precisely, the linear map
\[
\ker L(G)\longrightarrow\cB(G),
\qquad
x\longmapsto N^{\mathsf T}x,
\]
is surjective and has kernel $\langle\mathbf1\rangle$.
\end{proposition}

\begin{proof}
If $x\in\ker L(G)$, then
\[
N(N^{\mathsf T}x)=L(G)x=0,
\]
so $N^{\mathsf T}x \in \ker N=\cC(G)$. The vector $N^{\mathsf T}x$ belongs to the cut
space by construction (it is in the row space of $N$), hence $N^{\mathsf T}x  \in \cB(G)$, and the image of the map lies in $\cB(G)$ as desired.

Now choose $b\in\cB(G)$. Since $b$ is a cut it lies in the row space of $N$ and we can find $x$ so that
$b=N^{\mathsf T}x$. Since it is also a cycle,
\[
0=Nb=NN^{\mathsf T}x=L(G)x.
\]
So $b=N^{\mathsf T}x$ for some $x \in \ker L(G)$. Thus the map is surjective.

The kernel of the map is  $\ker N^{\mathsf T}$ by definition. A vector $x\in\F^V$ lies in this
kernel precisely when its values agree across every edge. Since $G$ is
connected, $x$ is constant. Hence the kernel is spanned by the constant vector $\langle\mathbf1\rangle$.
\end{proof}

Combining Proposition \ref{prop:lapbicycle} with Corollary \ref{cor:sigmabicycle} gives
\begin{equation}
\label{eq:lapformula}
\sigma(\lambda)
=
\nul_{\F}L(G_\lambda).
\end{equation}
This is also a special case of the medial-link nullity theorem of Silver and Williams \cite{SilverWilliams}.

\subsection{The graph of $2\times2$ blocks}

A $2\times2$ block in $[\lambda]$ is determined by its upper-left box
$(i,j)$, where
\[
1\leq i\leq\ell-1,
\qquad
1\leq j\leq\lambda_{i+1}-1.
\]

\begin{definition}
The graph $H_\lambda$ has one vertex for every $2\times2$ block of
$[\lambda]$. Two vertices are adjacent when the corresponding blocks
overlap in a $1\times2$ or $2\times1$ rectangle. Equivalently, their
upper-left boxes differ by one unit horizontally or vertically.
\end{definition}

Define
\begin{equation}
\label{eq:squarepartition}
\lambda^\square
:=
(\lambda_2-1,\lambda_3-1,\ldots,\lambda_\ell-1),
\end{equation}
with zero parts omitted. Then $\lambda^\square$ is the partition resulting from removing both the first row and column from the Young diagram of $\lambda$.  The possible upper-left boxes of the
$2\times2$ blocks form exactly this Young diagram. Therefore
\begin{equation}
\label{eq:Hshape}
H_\lambda= G_{\lambda^\square}.
\end{equation}
The number of vertices of $H_\lambda$ is
\begin{equation}
\label{eq:q}
q:=|V(H_\lambda)|
=
\sum_{i=2}^{\ell}(\lambda_i-1)
=
|\lambda|-\lambda_1-\ell+1.
\end{equation}

Every $2\times2$ block $Q$ determines a four-edge cycle $\partial Q$ in
$G_\lambda$. These are precisely the boundaries of the bounded faces in
the natural plane embedding of $G_\lambda$.

\begin{lemma}
\label{lem:facebasis}
The cycles
\[
\{\partial Q:Q\in V(H_\lambda)\}
\]
form a basis of $\cC(G_\lambda)$.
\end{lemma}

\begin{proof}
For any connected plane graph, the boundaries of the bounded faces form
a basis of the binary cycle space \cite[proof of Theorem 4.5.1]{Diestel}. In the natural embedding of
$G_\lambda$, the bounded faces are exactly the square faces corresponding
to the $2\times2$ blocks of $\lambda$. Hence their boundary cycles form
a basis.
\end{proof}
Consequently, for each coefficient vector
\[
z=(z_Q)_{Q\in V(H_\lambda)}
\in \F^{V(H_\lambda)},
\]
define
\begin{equation}
\label{eq:faceexpansion}
c(z)
=
\sum_{Q\in V(H_\lambda)} z_Q\,\partial Q.
\end{equation}
The map
\[
z\longmapsto c(z)
\]
is an isomorphism from $\F^{V(H_\lambda)}$ onto
$\cC(G_\lambda)$. Equivalently, every cycle in
$\cC(G_\lambda)$ has a unique expression of the form
\eqref{eq:faceexpansion}.

\begin{theorem}
\label{thm:facekernel}
The map $z\mapsto c(z)$ restricts to an isomorphism
\[
\ker_{\F}A(H_\lambda)
\cong
\cB(G_\lambda).
\]
Hence
\[
\dim_{\F}\cB(G_\lambda)
=
\nul_{\F}A(H_\lambda).
\]
\end{theorem}

\begin{proof}
A cycle $c(z)$ is a bicycle exactly when it is orthogonal to every cycle.
By Lemma~\ref{lem:facebasis}, it suffices to test orthogonality against
each face boundary $\partial R$.

The diagonal inner product is
\[
\langle\partial R,\partial R\rangle
=
|\partial R|
=
4
=
0
\qquad\text{in }\F.
\]
For distinct square faces $Q$ and $R$, their boundaries share one graph
edge exactly when the corresponding $2\times2$ blocks overlap in a
$1\times2$ or $2\times1$ strip. Thus
\[
\langle\partial Q,\partial R\rangle
=
\begin{cases}
1,&Q\sim R\text{ in }H_\lambda,\\
0,&\text{otherwise}.
\end{cases}
\]
It follows from \eqref{eq:faceexpansion} that
\[
\langle c(z),\partial R\rangle
=
\sum_{Q\sim R}z_Q.
\]
The collection of these equations, one for every $R$, is precisely the matrix equation $A(H_\lambda)z=0.$
Since the face expansion is an isomorphism from $\F^q$ onto
$\cC(G_\lambda)$, its restriction gives the claimed isomorphism.
\end{proof}

Combining Theorem~\ref{thm:facekernel} with
Corollary~\ref{cor:sigmabicycle} gives the simplest formula for the number of orbits, expressed only in terms of the nullity of an adjacency matrix.

\begin{corollary}
\label{cor:blockformula}
The number of billiard orbits in $\lambda$ is:
\begin{equation}
\label{eq:blockformula}
\sigma(\lambda)
=
1+\nul_{\F}A(H_\lambda).
\end{equation}
In particular we have a simple condition for there to be a unique orbit:
\[
\sigma(\lambda)=1
\quad\Longleftrightarrow\quad
A(G_{\lambda^\square})
\text{ is invertible over }\F.
\]
\end{corollary}
Corollary \ref{cor:blockformula} gives a  fast method for determining $\sigma(\lambda)$, as Gaussian elimination is extremely fast. Certainly it runs much faster than our original approach to count the orbits by computing them geometrically, although the geometric approach produced the nice diagrams illustrating the structure of the orbits, and gives information about the lengths of the orbits that the matrix approach does not.

\begin{example}
\label{ex:433blockformula}
Let $\lambda=(4,3,3)$. We compute the number of billiard orbits using
Corollary~\ref{cor:blockformula}. Since $\lambda^\square=(2,2)$ we have $H_\lambda$ is the cell-adjacency graph $G_{(2,2)}$.
This graph is simply a four-cycle with
adjacency matrix
\[
A(H_\lambda)
=
\begin{pmatrix}
0&1&0&1\\
1&0&1&0\\
0&1&0&1\\
1&0&1&0
\end{pmatrix}
\]
which  has rank two and nullity two.
Applying \eqref{eq:blockformula}, we obtain
\[
\sigma(4,3,3)
=
1+\nul_{\F}A(H_\lambda)
=
1+2
=
3,
\] which matches Figure~\ref{fig:orbitexamples}(c).
\end{example}

\begin{example}
\label{ex:431blockformula}
Let $\lambda=(4,3,1)$ so $
\lambda^\square=
(2)$. Then $H_\lambda$ is simply two vertices joined by a single edge so has adjacency matrix
$\begin{pmatrix}
0&1\\
1&0
\end{pmatrix}.$ Over $\F$, this matrix is invertible, so $\sigma(4,3,1)=1$.
\end{example}

\begin{remark}
If $\lambda$ has no $2\times2$ block, then $H_\lambda$ is empty. Its
$0\times0$ adjacency matrix has nullity zero, and
\eqref{eq:blockformula} gives $\sigma(\lambda)=1$. Such partitions are
exactly the hooks $(a,1^b)$, which are easily seen to have a single orbit.
\end{remark}

Using this linear algebra approach we computed  Table \ref{tab:orbitdistribution} which shows the distribution of the values $\sigma(\lambda)$:

\begin{table}[ht]
\centering
\caption{The number of partitions $\lambda\vdash n$ with each orbit
number $\sigma(\lambda)$, for $1\leq n\leq 25$. Each row sums to the
number of partitions $p(n)$.}
\label{tab:orbitdistribution}
\begin{minipage}[t]{0.48\textwidth}
\centering
\begin{tabular}{r rrrrr r}
\toprule
$n$ & $1$ & $2$ & $3$ & $4$ & $5$ & $p(n)$ \\
\midrule
1  & 1   & 0   & 0  & 0 & 0 & 1    \\
2  & 2   & 0   & 0  & 0 & 0 & 2    \\
3  & 3   & 0   & 0  & 0 & 0 & 3    \\
4  & 4   & 1   & 0  & 0 & 0 & 5    \\
5  & 5   & 2   & 0  & 0 & 0 & 7    \\
6  & 8   & 3   & 0  & 0 & 0 & 11   \\
7  & 11  & 4   & 0  & 0 & 0 & 15   \\
8  & 14  & 8   & 0  & 0 & 0 & 22   \\
9  & 17  & 12  & 1  & 0 & 0 & 30   \\
10 & 24  & 16  & 2  & 0 & 0 & 42   \\
11 & 31  & 22  & 3  & 0 & 0 & 56   \\
12 & 40  & 33  & 4  & 0 & 0 & 77   \\
13 & 49  & 44  & 8  & 0 & 0 & 101  \\
\bottomrule
\end{tabular}
\end{minipage}
\hfill
\begin{minipage}[t]{0.48\textwidth}
\centering
\begin{tabular}{r rrrrr r}
\toprule
$n$ & $1$ & $2$ & $3$ & $4$ & $5$ & $p(n)$ \\
\midrule
14 & 64  & 59  & 12  & 0  & 0 & 135  \\
15 & 81  & 78  & 17  & 0  & 0 & 176  \\
16 & 104 & 104 & 22  & 1  & 0 & 231  \\
17 & 127 & 136 & 32  & 2  & 0 & 297  \\
18 & 160 & 176 & 46  & 3  & 0 & 385  \\
19 & 197 & 224 & 65  & 4  & 0 & 490  \\
20 & 246 & 289 & 84  & 8  & 0 & 627  \\
21 & 299 & 366 & 115 & 12 & 0 & 792  \\
22 & 370 & 461 & 154 & 17 & 0 & 1002 \\
23 & 449 & 580 & 204 & 22 & 0 & 1255 \\
24 & 552 & 728 & 262 & 33 & 0 & 1575 \\
25 & 663 & 904 & 344 & 46 & 1 & 1958 \\
\bottomrule
\end{tabular}
\end{minipage}
\end{table}

The reader will  observe the first nonzero entry in column $k$ is a one at $n=k^2$, corresponding to the square partition $(k^k)$. This is no coincidence.  Recall that the \emph{Durfee length} of $\lambda$ is the side length of the largest square (the \emph{Durfee square}) that fits inside $[\lambda]$. Alternately it is the largest $k$ so that $\lambda_k \geq k$. The following result was communicated to us by Aidan Botkin
\cite{BotkinPrivate}. We give a linear-algebraic proof using the
adjacency-nullity formula.

\begin{theorem}\cite{BotkinPrivate}
\label{thm:durfee-bound}
Let $\lambda$ be a partition whose Durfee square has side length
$d$. Then
$
\sigma(\lambda)\leq d.$
\end{theorem}

\begin{proof}
Put $\mu=\lambda^\square$ so $\mu$ has Durfee length $r:=d-1$.
If $r=0$, then $\lambda$ is a hook partition with one orbit, so the result holds.
So assume $r\geq1$.

For $1\leq i\leq r$,
let $s_i=(i,\mu_i)$
be the rightmost box in row $i$ of the Young diagram $[\mu]$. These are also the northeast endpoints
of the $r$ diagonal hooks of $\mu$.

We claim that a vector
\[
x\in\ker_{\mathbb F_2}A(G_\mu)
\]
is uniquely determined by the $r$ coordinates
$x_{s_1},\ldots,x_{s_r}.
$
For convenience, set $x_{i,j}=0$ whenever $(i,j)$ is not a box
of $\mu$. The equation corresponding to the box $(i,j)$ is then
\[
x_{i-1,j}+x_{i+1,j}+x_{i,j-1}+x_{i,j+1}=0.
\]
So we assume the $x_{s_i}$ are all specified. We first show how to determine all coordinates $x_{i,j}$ on or above the main diagonal, i.e. with $i \leq j$. This is done column by column, from right to left. Then we use these coordinates to proceed row by row top to bottom to determine the $x_{i,j}$ with $i>j.$

The boxes in the rightmost column lying on or above the main diagonal are all among the $s_i$. So we may suppose that all coordinates in columns strictly to the right of
column $j$ have already been determined. Consider a box $(i,j)$
with
\[
1\leq i\leq j.
\]
If $j=\mu_i$, then $(i,j)=s_i$, so its coordinate is one of the
prescribed coordinates. If $j<\mu_i$, then the box $(i,j+1)$
belongs to $\mu$, and its adjacency equation gives
\[
x_{i,j}
=
x_{i,j+2}
+x_{i-1,j+1}
+x_{i+1,j+1}.
\]
Every coordinate on the right-hand side lies in column $j+1$ or
column $j+2$, and hence has already been determined. Thus, at each
stage, we determine one new box in every applicable row. Continuing
column by column determines all $x_{i,j}$ with $i\leq j$.

We next determine the coordinates below the main diagonal. This is
done row by row, from top to bottom, in the exact same manner.

Consequently, every coordinate of $x$ is uniquely determined by
the $r$ values at the northeast endpoints of the diagonal hooks.
It follows that
\[
\nul_{\mathbb F_2}A(G_\mu)\leq r=d-1.
\]
Using the adjacency-nullity formula, we obtain
\[
\sigma(\lambda)
=
1+\nul_{\mathbb F_2}A(G_\mu)
\leq
1+(d-1)
=
d.
\]
\end{proof}

\begin{remark} We believe this result illustrates the power of the linear algebra method, proving each trajectory passes through the Durfee square is doable but difficult to make precise.
\end{remark}

\begin{corollary}
\label{cor:orbitslessthansqrtn}
Suppose $\lambda$ is a partition of $n$. Then $\sigma(\lambda) \leq \sqrt{n}.$
\end{corollary}

\section{Rectangular partitions and Fibonacci polynomials}
The formula  $\sigma(n^m)=\gcd(m,n)$ appears in \cite{Gerdes1999} and may be derived directly from the orbit diagram. In this section we prove it with our adjacency matrix approach. We find some interesting algebra involving Fibonacci polynomials which we later speculate may generalize to more complicated partitions. So  let $\lambda=(n^m)$. The graph $H_\lambda$ is just an $(m-1)\times (n-1)$ grid, which is a graph Cartesian product of paths on $m-1$ and $n-1$ vertices. We need to determine the nullity of the adjacency matrix of this graph.

We begin by recalling the \emph{Fibonacci polynomials}, defined recursively over $\F$ by
\begin{equation}
\label{eq:recursivedefFibonacciPoly}
p_0(t)=1,
\qquad
p_1(t)=t,
\qquad
p_k(t)=tp_{k-1}(t)+p_{k-2}(t).
\end{equation}
The greatest common divisors of these polynomials over the ring $\F[t]$ satisfy a nice property we will need later, namely that they are also Fibonacci polynomials:

\begin{lemma}
\label{lem:fibgcd}
For all $r,s\geq 0$,
\[
\gcd\bigl(p_r(t),p_s(t)\bigr)
=
p_{\gcd(r+1,s+1)-1}(t),
\]
\end{lemma}

\begin{proof}
Set $F_n=p_{n-1}$ for $n\geq 1$, with $F_0=0$. Then
\[
F_{n+1}=tF_n+F_{n-1}.
\]
The standard addition identity
\[
F_{a+b}=F_aF_{b+1}+F_{a-1}F_b
\]
and the fact that consecutive Fibonacci polynomials are coprime implies,
for $m>n$,
\[
\gcd(F_m,F_n)=\gcd(F_{m-n},F_n).
\]
The Euclidean algorithm on the indices therefore gives
\[
\gcd(F_m,F_n)=F_{\gcd(m,n)}.
\]
Taking $m=r+1$ and $n=s+1$ yields the result.
\end{proof}

These polynomials arise as characteristic polynomials of the adjacency matrices for the paths, remembering again that we are working over $\F$:

\begin{proposition}
  \label{prop:Fibonaccipolynomialequalcharapoly}
  Let $A_s=A(P_s)$ be the adjacency matrix of the path $P_s$ on $s$ vertices. Then the characteristic polynomial of $A_s$ is $p_s(t)$.
\end{proposition}
\begin{proof}

Ordering the vertices along the path and using that $-1=1$ in $\F$, we have
\[
tI_s-A_s=
\begin{pmatrix}
t      & 1      & 0      & \cdots & 0      \\
1      & t      & 1      & \ddots & \vdots \\
0      & 1      & t      & \ddots & 0      \\
\vdots & \ddots & \ddots & \ddots & 1      \\
0      & \cdots & 0      & 1      & t
\end{pmatrix}.
\]
Let
\[
D_s(t)=\det(tI_s-A_s).
\]
Expanding this determinant along the first column gives
\[
D_s(t)=tD_{s-1}(t)+D_{s-2}(t).
\]
The initial values are
\[
D_1(t)=t,
\qquad
D_2(t)=t^2+1.
\]
Consequently, $D_s(t)$ satisfies the same recurrence and initial conditions as
 $p_s(t)$, giving the result.
\end{proof}

Now for $\lambda=(n^m)$ we have
\[
G_\lambda=P_m\square P_n,
\qquad
H_\lambda=P_{m-1}\square P_{n-1},
\]
where $\square$ denotes the Cartesian product of graphs.
To apply Corollary~\ref{cor:blockformula} we need to compute the nullity of the adjacency matrix of the Cartesian product of two paths:
\begin{theorem}
\label{thm:gridnullity}
For positive integers $r,s$,
\[
\nul_{\F}A(P_r\square P_s)
=
\gcd(r+1,s+1)-1.
\]
Consequently,
\[
\sigma(n^m)=\gcd(m,n).
\]
\end{theorem}

\begin{proof}
The graph $P_r\square P_s$ is a rectangular grid of $r$ rows and $s$ columns. If we order the vertices along each successive row then the adjacency matrix has a simple block form, each block row has a copy of $A_s$ for the edges within the row and then identity matrices reflecting the vertical edges to the previous and following rows (or zero for the top and bottom row). For example when $r=5$ and $s=7$ we have:
\[
A(P_5\square P_7)
=
\begin{pmatrix}
A_7 & I_7 & 0   & 0   & 0   \\
I_7 & A_7 & I_7 & 0   & 0   \\
0   & I_7 & A_7 & I_7 & 0   \\
0   & 0   & I_7 & A_7 & I_7 \\
0   & 0   & 0   & I_7 & A_7
\end{pmatrix}.
\]

Regard a column vector acted on by  $A(P_r\square P_s)$ as a sequence
$x_1,\ldots,x_r\in\F^s$. So in our example above, multiplying by $A(P_5\square P_7)$ on the left we would get
$$(A_7x_1+x_2,x_1+A_7x_2+x_3, x_2+A_7x_3+x_4,x_3+A_7x_4+x_5,x_4+A_7x_5).$$
In the general case choosing
$x=(x_1,\ldots,x_r)\in\ker_{\F} A(P_r\square P_s)
$
gives a system of equations:
\[
x_{i-1}+A_sx_i+x_{i+1}=0
\]
where we set $x_0=x_{r+1}=0$.
Since we are working over $\F$,
this may equivalently be written as
\begin{equation}
\label{eq:gridrecurrence}
x_{i+1}=A_sx_i+x_{i-1}.
\end{equation}

The recursion \eqref{eq:gridrecurrence} implies a vector in the kernel is entirely determined by $x_1$.  Starting with $x_0=0$ and $x_1=u$, induction on $i$ gives
\[
x_i=p_{i-1}(A_s)u
\qquad
(1\leq i\leq r+1).
\]

The final boundary condition $x_{r+1}=0$ is therefore equivalent to
\begin{equation}
\label{eq:prcondition}
p_r(A_s)u=0.
\end{equation}
Conversely, every $u\in\ker (p_r(A_s))$ determines, by the recurrence
above, a unique vector
\[
x=
\bigl(
u,\,
p_1(A_s)u,\,
p_2(A_s)u,\,
\ldots,\,
p_{r-1}(A_s)u
\bigr) \in\ker A(P_r\square P_s).
\]
Thus the map $x\mapsto x_1$ gives an isomorphism
\begin{equation}
\label{eq:kernelofcartesianproductequalsnullspace}
\ker A(P_r\square P_s)
\cong
\ker p_r(A_s),
\end{equation}
and consequently
\[
\nul_{\F}A(P_r\square P_s)
=
\dim_{\F}\ker p_r(A_s).
\]

It now remains to compute the kernel of $p_r(A_s)$. The matrix $A_s$ has a simple structure with ones on the super and subdiagonals, for example:
\[
A_6=
\begin{pmatrix}
0&1&0&0&0&0\\
1&0&1&0&0&0\\
0&1&0&1&0&0\\
0&0&1&0&1&0\\
0&0&0&1&0&1\\
0&0&0&0&1&0
\end{pmatrix},
\]
from which we easily see the vectors
\[
e_1,\ A_se_1,\ A_s^2e_1,\ldots,A_s^{s-1}e_1
\]
form a basis of $\F^s$. Relative to the standard basis, their
coordinate matrix is triangular with every diagonal entry equal to
$1$. It follows that the minimal polynomial of $A_s$ has degree $s$.
Since the minimal polynomial divides the characteristic polynomial,
and since the characteristic polynomial of $A_s$ is $p_s(t)$ of
degree $s$, the minimal polynomial is also $p_s(t)$. Hence the map
\[
\F[t]\longrightarrow\F^s,
\qquad
f(t)\longmapsto f(A_s)e_1,
\]
induces an isomorphism of $\F[t]$-modules
\[
\F[t]/(p_s)\cong\F^s,
\]
under which multiplication by $t$ corresponds to the action of
$A_s$.

Under this identification, the operator $p_r(A_s)$ corresponds to
multiplication by $p_r(t)$ on $\F[t]/(p_s)$. Let
\[
d(t)=\gcd\bigl(p_r(t),p_s(t)\bigr),
\]
and write
\[
p_r(t)=d(t)a(t),
\qquad
p_s(t)=d(t)b(t),
\]
where
\[
\gcd\bigl(a(t),b(t)\bigr)=1.
\]

A residue class $f(t)+(p_s)$ lies in the kernel precisely when
\[
p_s(t)\mid p_r(t)f(t).
\]
Since
\[
p_s(t)=d(t)b(t)
\qquad\text{and}\qquad
p_r(t)=d(t)a(t),
\]
this is equivalent to
\[
b(t)\mid a(t)f(t).
\]
Because $a(t)$ and $b(t)$ are coprime, this holds precisely when
\[
b(t)\mid f(t).
\]

Therefore the kernel is the ideal of $\F[t]/(p_s)$ generated by
\[
b(t)=\frac{p_s(t)}{d(t)}.
\]
Its elements are represented uniquely by
\[
h(t)b(t),
\qquad
\deg h<\deg d.
\]
Consequently,
\begin{equation}
\label{eq:gcdnull}
\dim_{\F}\ker p_r(A_s)
=
\deg d
=
\deg\gcd(p_r,p_s).
\end{equation}

Since $\deg p_k=k$, equation~\eqref{eq:gcdnull} and
Lemma~\ref{lem:fibgcd} yield
\[
\begin{aligned}
\nul_{\F}A(P_r\square P_s)
&=\deg\gcd(p_r,p_s)\\
&=\deg p_{\gcd(r+1,s+1)-1}\\
&=\gcd(r+1,s+1)-1.
\end{aligned}
\]
If $m=1$ or $n=1$, then $(n^m)$ is a hook partition, so
\[
\sigma(n^m)=1=\gcd(m,n).
\]
We may therefore assume that $m,n\geq2$.
Taking $r=m-1$ and $s=n-1$ in \eqref{eq:blockformula} gives
\[
\sigma(n^m)
=
1+\gcd(m,n)-1
=
\gcd(m,n).
\]
Thus a rectangular Young diagram has one orbit precisely when its two
side lengths are coprime.
\end{proof}
\section{Odd domino tilings and the one-orbit condition}

The nonsingularity criterion for $A(H_\lambda)$ has a useful enumerative
interpretation, possible because we are working over a field of characteristic two, and thus the determinant of a matrix is the same as its permanent. Recall that a \emph{perfect matching} of a graph is a subset of the edges such that each vertex lies in precisely one edge. In particular these cannot exist if the number of vertices is odd. Perfect matchings of $G_\lambda$ are clearly the same as domino tilings of the Young diagram $[\lambda].$ We adopt the convention that the empty Young diagram has one domino tiling, namely the empty tiling. Equivalently, the empty graph has one perfect matching.

\begin{theorem}
\label{thm:oddtiling}
Let $\lambda$ be a partition. Then $\sigma(\lambda)=1$ if and only if $\lambda^\square$ has an odd number of domino tilings.
\end{theorem}

\begin{proof}
Set $\mu=\lambda^\square$ and color $G_{\mu}$ black or white by the parity of $i+j$ and order the vertices with all
black vertices first. Its adjacency matrix has block form
\[
A(G_\mu)
=
\begin{pmatrix}
0&B\\
B^{\mathsf T}&0
\end{pmatrix}.
\]
If the two color classes have different sizes, the matrix is singular
and there is no perfect matching. Since the nullity is therefore positive, Corollary \ref{cor:blockformula} implies $\sigma(\lambda)>1$.

Otherwise, if both classes have the same size
then over $\F$,
\[
\det A(G_\mu)
=
\det(B)^2
=
\det(B),
\]
since $\det(B) \in \{0,1\}.$
Since signs disappear in characteristic two,
\[
\det(B)
=
\operatorname{perm}(B)
\pmod2.
\]

The permanent of $B$ is the number of perfect matchings of $G_\mu$
\cite[Proposition~9.3.1]{AsratianDenleyHaggkvist}.
So $A(G_\mu)$ is singular precisely when the number of perfect matchings is even. The result follows from
Corollary~\ref{cor:blockformula}.
\end{proof}

This criterion concerns parity, not the exact number of orbits. If the number of tilings is odd there is exactly one orbit; if it is even there is some number greater than one, which the criterion does not determine.
Nevertheless, it converts the one-orbit problem into a familiar
perfect-matching question; see Lov\'asz--Plummer~\cite{LovaszPlummer}
for background on perfect matchings.

\begin{example}
\label{ex:tilingparity}
The following examples illustrate the criterion:
\[
\begin{array}{c|c|c|c}
\lambda&
\lambda^\square&
\text{number of domino tilings}&
\sigma(\lambda)
\\ \hline
(4,3,1)&(2)&1&1\\
(4,4,4)&(3,3)&3&1\\
(3,3,3)&(2,2)&2&3\\
(5,4,4,4)&(3,3,3)&0&4
\end{array}
\]
For example, the $2\times3$ rectangle $(3,3)$ has three domino tilings,
so $A(G_{(3,3)})$ is nonsingular over $\F$ and $(4,4,4)$ has one
billiard orbit. The $2\times2$ square has two tilings, so the associated
partition $(3,3,3)$ does not have one orbit.
\end{example}

For a $2\times r$ rectangle, the number of domino tilings is the
ordinary Fibonacci number $f_{r+1}$. Its parity has period three. This
is consistent with the rectangular orbit formula: if
\[
\lambda=(r+1,r+1,r+1),
\]
then
\[
\lambda^\square=(r,r)
\]
and
\[
\sigma(\lambda)=1
\quad\Longleftrightarrow\quad
3\nmid r+1
\quad\Longleftrightarrow\quad
f_{r+1}\text{ is odd}.
\]

Related connections among square-grid billiards, kernels of adjacency
matrices over $\mathbb F_2$, and the $2$-divisibility of perfect-matching
numbers were developed by Barkley and Liu~\cite{BarkleyLiu}. Their
``billiard-nest'' construction is different from the closed billiard orbits
considered here, in particular their trajectories may branch at points. The two theories share the appearance of binary
adjacency kernels and domino-tiling parity.
\section{Checkerboard imbalance and rank deficiency}
\label{sec:checkerboardrank}

Let $\mu=\lambda^\square,$
and color the boxes of $\mu$ as a checkerboard, with the northwest
box black. Write
\[
b(\mu)=\#\{\text{black boxes of }\mu\},
\qquad
w(\mu)=\#\{\text{white boxes of }\mu\}.
\]
Since the cell-adjacency graph $G_\mu$ is bipartite, with the black
and white boxes forming its two vertex classes, it has a
$b(\mu)\times w(\mu)$ biadjacency matrix $B_\mu$. Its rows are
indexed by the black boxes, its columns by the white boxes, and
\[
(B_\mu)_{xy}
=
\begin{cases}
1,&\text{if the boxes $x$ and $y$ share a side},\\
0,&\text{otherwise}.
\end{cases}
\]
We shall use the two quantities
\[
\delta(\mu)=|b(\mu)-w(\mu)|
\]
and
\[
\epsilon(\mu)
=
\min\{b(\mu),w(\mu)\}
-\rank_{\F}(B_\mu).
\]
Thus $\delta(\mu)$ measures the imbalance between white and black vertices, while
$\epsilon(\mu)$ measures the deficiency of $B_\mu$ from having
the largest possible rank, which is $\min\{b(\mu),w(\mu)\}$.

\begin{remark}
  \label{rmk:BG}  invariant $\delta(\mu)$ is the absolute value of the so-called BG-rank defined in \cite{BerkovichGarvan2006}.
\end{remark}

\subsection{The rank decomposition}

\begin{proposition}
\label{prop:checkerboardrank}
For every partition $\lambda$, with $\mu=\lambda^\square$,
\[
{
\sigma(\lambda)
=
1+b(\mu)+w(\mu)
-2\rank_{\F}(B_\mu).
}
\]
Equivalently,
\[
{
\sigma(\lambda)
=
1+\delta(\mu)+2\epsilon(\mu).
}
\]
\end{proposition}

\begin{proof}
Order the vertices of $G_\mu$ with the black boxes first and the
white boxes second. Since adjacent boxes have opposite colors,
\[
A(G_\mu)
=
\begin{pmatrix}
0&B_\mu\\
B_\mu^{\mathsf T}&0
\end{pmatrix},
\]
and so
\[
\rank_{\F}A(G_\mu)
=
\rank_{\F}B_\mu
+
\rank_{\F}B_\mu^{\mathsf T}
=
2\rank_{\F}B_\mu.
\]
Consequently,
\[
\nul_{\F}A(G_\mu)
=
b(\mu)+w(\mu)-2\rank_{\F}B_\mu.
\]
The first formula now follows from
Corollary~\ref{cor:blockformula}.

For the second formula, suppose without loss of generality that
$b(\mu)\geq w(\mu)$. Then
\[
\begin{aligned}
b(\mu)+w(\mu)-2\rank_{\F}B_\mu
&=
b(\mu)-w(\mu)
+2\bigl(w(\mu)-\rank_{\F}B_\mu\bigr)\\
&=
\delta(\mu)+2\epsilon(\mu).
\end{aligned}
\]
\end{proof}

\begin{remark}
Proposition \ref{prop:checkerboardrank} separates the excess orbits beyond one into two distinct
contributions, which we call the ``checkerboard imbalance" and the ``rank deficiency":
\[
\sigma(\lambda)-1
=
\underbrace{\delta(\mu)}_{\text{checkerboard imbalance}}
+
\underbrace{2\epsilon(\mu)}_{\text{rank deficiency}}.
\]
The first is visible directly from the checkerboard coloring, and is the absolute value of the BG-rank. The
second measures additional orbit structure that cannot be detected
merely by counting the two colors.
\end{remark}
\subsection{Consequences for the orbit number}

The size of the truncated diagram determines the parity of the orbit
number.

\begin{corollary}
\label{cor:orbitparity}
For every partition $\lambda$,
\[
\sigma(\lambda)
\equiv
|\lambda^\square|+1
\pmod 2.
\]
In particular, if $|\lambda^\square|$ is odd, then
$\sigma(\lambda)$ is even and hence cannot equal $1$.
\end{corollary}

\begin{proof}
Reducing the first formula in
Proposition~\ref{prop:checkerboardrank} modulo $2$, and using
\[
b(\lambda^\square)+w(\lambda^\square)
=
|\lambda^\square|,
\]
gives the result.
\end{proof}

The checkerboard imbalance also gives a  lower bound on the number of orbits:

\begin{corollary}
\label{cor:checkerboardlowerbound}
For every partition $\lambda$, with $\mu=\lambda^\square$,
\[
{
\sigma(\lambda)\geq 1+|b(\mu)-w(\mu)|.
}
\]
Equality holds if and only if $B_\mu$ has full rank, that is, if and only if
\[
\rank_{\F}(B_\mu)=\min\{b(\mu),w(\mu)\}.
\]
\end{corollary}

\begin{proof}
By Proposition~\ref{prop:checkerboardrank},
\[
\sigma(\lambda)
=
1+\delta(\mu)+2\epsilon(\mu),
\]
and $\epsilon(\mu)\geq0$. Equality holds precisely when
$\epsilon(\mu)=0$, which is equivalent to $B_\mu$ having full
rank.
\end{proof}

The checkerboard imbalance is particularly easy to compute directly from
the parts of $\mu$. A row of even length contains equally many black and white
boxes, whereas an odd row contributes $+1$ or $-1$ according as its row
index is odd or even. Thus, if
\[
\mu=(\mu_1,\ldots,\mu_r),
\]
then
\[
b(\mu)-w(\mu)
=
\sum_{\substack{1\leq i\leq r\\ \mu_i\text{ odd}}}
(-1)^{i+1}.
\]
Consequently, Corollary~\ref{cor:checkerboardlowerbound} gives the
explicit bound
\[
\sigma(\lambda)
\geq
1+
\left|
\sum_{\substack{i\geq1\\
(\lambda^\square)_i\text{ odd}}}
(-1)^{i+1}
\right|.
\]
The signed quantity $b(\mu)-w(\mu)$ is the BG-rank of $\mu$,
so the checkerboard lower bound may also be expressed in terms of
this classical partition statistic.

More precisely, the orbit number is constrained to lie in the
arithmetic progression
\[
1+\delta(\mu),\quad
3+\delta(\mu),\quad
5+\delta(\mu),\quad\ldots.
\]
Thus the checkerboard coloring determines both a lower bound and a
parity restriction on $\sigma(\lambda)$.

The first two possible orbit numbers admit particularly simple
characterizations.

\begin{corollary}
\label{cor:oneorbitcheckerboard}
Let $\mu=\lambda^\square$.

\begin{enumerate}
\item
The partition $\lambda$ has one orbit if and only if
\[
b(\mu)=w(\mu)
\]
and $B_\mu$ has full rank over $\F$.

\item
The partition $\lambda$ has two orbits if and only if
\[
|b(\mu)-w(\mu)|=1
\]
and $B_\mu$ has full rank.
\end{enumerate}

Here the empty $0\times0$ matrix is regarded as full rank.
\end{corollary}

\begin{proof}
By Proposition~\ref{prop:checkerboardrank},
\[
\sigma(\lambda)=1+\delta(\mu)+2\epsilon(\mu).
\]
Thus $\sigma(\lambda)=1$ precisely when
\[
\delta(\mu)=\epsilon(\mu)=0.
\]
The first equality says that $b(\mu)=w(\mu)$, and the second then
says that the square matrix $B_\mu$ is nonsingular.

Similarly, $\sigma(\lambda)=2$ precisely when
\[
\delta(\mu)+2\epsilon(\mu)=1.
\]
Since both terms are nonnegative integers and the second is even, this
is equivalent to
\[
\delta(\mu)=1,
\qquad
\epsilon(\mu)=0.
\]
\end{proof}

\subsection{Imbalance does not determine the orbit number}

The rank-deficiency term is essential: two truncated diagrams can
have the same size and checkerboard imbalance but different orbit
numbers.

\begin{example}
Consider
$\lambda=(5,5)$ and $\nu=(3,3,3).
$
Their truncated diagrams are
\[
\lambda^\square=(4),
\qquad
\nu^\square=(2,2).
\]
Each contains two black boxes and two white boxes, so
\[
\delta(\lambda^\square)
=
\delta(\nu^\square)
=
0.
\]

For $\lambda^\square=(4)$, the biadjacency matrix may be written as
\[
B_{\lambda^\square}
=
\begin{pmatrix}
1&0\\
1&1
\end{pmatrix}.
\]
This matrix has full rank over $\F$, so
$
\epsilon(\lambda^\square)=0.
$
Proposition~\ref{prop:checkerboardrank} therefore gives
$
\sigma(\lambda)=1.
$

For $\nu^\square=(2,2)$, the biadjacency matrix is
\[
B_{\nu^\square}
=
\begin{pmatrix}
1&1\\
1&1
\end{pmatrix}.
\]
This matrix has rank $1$, and hence
$
\epsilon(\nu^\square)=2-1=1.
$
It follows that
\[
\sigma(\nu)
=
1+\delta(\nu^\square)+2\epsilon(\nu^\square)
=
3.
\]

Thus the checkerboard lower bound is attained for $\lambda=(5,5)$
but not for $\nu=(3,3,3)$, even though their truncated diagrams
have the same size and the same checkerboard imbalance.
\end{example}
The preceding example shows that the imbalance $\delta(\mu)$ can
underestimate $\sigma(\lambda)$, the discrepancy being the
rank-deficiency term $2\epsilon(\mu)$. It is therefore natural to ask
for families in which no such discrepancy occurs, so that the
elementary bound of Corollary~\ref{cor:checkerboardlowerbound} is sharp.
The staircase partitions form one natural such family, and for them the
orbit count admits a closed formula as clean as the rectangular one, by
a completely different mechanism. We present this proof to highlight the tools we develop, one can also explicitly work out the trajectories inside the staircase partitions.
\subsection{The staircase partition}
\label{subsec:staircase}

Throughout this subsection we use box coordinates $(i,j)$ with
$i,j\geq1$, so that $(i,j)\in\rho_m$ if and only if $i+j\leq m+1$, where
\[
\rho_m=(m,m-1,\ldots,2,1).
\]
The \emph{antidiagonal} of a box is $d(i,j)=i+j$, which ranges over
$\{2,3,\ldots,m+1\}$.

\begin{lemma}
\label{lem:staircaseselfsimilar}
For $n\geq3$ we have $\rho_n^{\square}=\rho_{n-2}$, where
$\rho_0=\rho_{-1}=\varnothing$.
\end{lemma}

\begin{proof}
By \eqref{eq:squarepartition},
$\lambda^{\square}=(\lambda_2-1,\lambda_3-1,\ldots)$ with zero parts
omitted. For $\lambda=\rho_n$ this is
$(n-2,n-3,\ldots,1,0)=\rho_{n-2}$.
\end{proof}

We need the nullity of the adjacency matrix. We first record that it has a nice antidiagonal structure.

\begin{lemma}
\label{lem:staircaseantidiag}
Color the box $(i,j)$ black when $i+j$ is even and white when $i+j$ is
odd, so that the northwest box $(1,1)$ is black. Then:
\begin{enumerate}
\item Antidiagonal $d$ consists of the boxes $(i,d-i)$ for
$1\leq i\leq d-1$; it has exactly $d-1$ boxes, all of the color given by
the parity of $d$.
\item Every edge of $G_{\rho_m}$ joins a box on antidiagonal $d$ to a
box on antidiagonal $d+1$. Consequently the black and white classes are
the even and odd antidiagonals, respectively.
\end{enumerate}
\end{lemma}

\begin{proof}
For (1), the conditions $i\geq1$, $d-i\geq1$, and $i+(d-i)=d\leq m+1$
give $1\leq i\leq d-1$, hence $d-1$ boxes, each of color parity
$d$. For (2), adjacent boxes differ by one in a single coordinate, so
their antidiagonals differ by one; in particular adjacent boxes receive
opposite colors.
\end{proof}
The first antidiagonal has one black box, the second has two white boxes, the third three black boxes, etc. So the following is immediate:
\begin{lemma}
\label{lem:staircaseimbalance}
With $b(\rho_m)$ and $w(\rho_m)$ the numbers of black and white
boxes,
\[
\delta(\rho_m)
=
\bigl|\,b(\rho_m)-w(\rho_m)\,\bigr|
=
\Bigl\lceil\tfrac{m}{2}\Bigr\rceil,
\]
and the majority color is the color of the boxes on antidiagonal $m+1$.
\end{lemma}

We will prove the biadjacency matrix of $\rho_m$
has full rank by exhibiting an explicit square submatrix
that is triangular. For our rows we choose every box of minority color and for the columns we choose the boxes to their immediate right. Ordering them properly we get a lower unitriangular matrix.

Let $M$ denote the minority color class of $\rho_m$. Define
\[
\phi\colon M\longrightarrow[\rho_m],
\qquad
\phi(i,j)=(i,j+1).
\]

\begin{lemma}
\label{lem:staircasephi}
The map $\phi$ is well defined, and $\phi(x)$ has the majority color for
every $x\in M$.
\end{lemma}

\begin{proof}
Let $x=(i,j)\in M$ have antidiagonal $d=i+j$, necessarily of minority
parity. By Lemma~\ref{lem:staircaseimbalance} the majority color is that
of antidiagonal $m+1$, so $d\not\equiv m+1\pmod2$; in particular $d\leq m$.
Hence $i+(j+1)=d+1\leq m+1$, so $\phi(x)=(i,j+1)\in[\rho_m]$. Its
antidiagonal is $d+1$, of majority parity.
\end{proof}

\begin{lemma}
\label{lem:staircasefullrank}
The biadjacency matrix $B_{\rho_m}$ has rank
$\min\{b(\rho_m),w(\rho_m)\}$ over $\F$. Equivalently,
$\epsilon(\rho_m)=0$.
\end{lemma}

\begin{proof}
Recall that the rows of $B_{\rho_m}$ are indexed by the black boxes and
its columns by the white boxes. Write $C$ for the matrix $B_{\rho_m}$ if
the minority class $M$ is black, and $B_{\rho_m}^{\mathsf T}$ if $M$ is
white; in either case the rows of $C$ are indexed by $M$, the columns by
the majority class, and $\rank_{\F}C=\rank_{\F}B_{\rho_m}$.

Order the boxes of $M$ by column, and within each column by row; that is,
$(i,j)\prec(i',j')$ when $j<j'$, or when $j=j'$ and $i<i'$. List
$M=\{x_1\prec x_2\prec\cdots\prec x_N\}$ with $N=|M|$. Let
$S\in\F^{N\times N}$ be the submatrix of $C$ with rows indexed
by $x_1,\ldots,x_N$ and columns by $\phi(x_1),\ldots,\phi(x_N)$, so that
$S_{rs}=1$ precisely when $x_r$ is adjacent to $\phi(x_s)$.

Write $x_s=(c,e)$, so $\phi(x_s)=(c,e+1)$. The boxes adjacent to
$(c,e+1)$ are its four grid neighbors,
\[
(c,e),\qquad (c,e+2),\qquad (c-1,e+1),\qquad (c+1,e+1),
\]
so $S_{rs}=1$ if and only if $x_r$ is one of these four boxes.

The box $(c,e)$ is $x_s$ itself, and it is the left neighbor of
$\phi(x_s)$; hence $S_{ss}=1$ and $S$ has unit diagonal. Now suppose
$S_{rs}=1$ with $r\neq s$, so that $x_r$ is one of the remaining three
neighbors. The box $x_s=(c,e)$ lies in column $e$, while each of these
three neighbors lies in a strictly later column:
\[
(c,e+2)\ \text{in column }e+2,\qquad
(c\pm1,e+1)\ \text{in column }e+1.
\]
Since our order lists boxes by increasing column, $x_r$ appears strictly
after $x_s$, that is $r>s$. Hence every off-diagonal nonzero entry of
$S$ lies strictly below the diagonal.

Thus $S$ is lower unitriangular, so
$\det_{\F}S=1$ and $S$ is nonsingular. Its $N$ rows are linearly
independent rows of $C$, so
$\rank_{\F}B_{\rho_m}=\rank_{\F}C\geq N=\min\{b(\rho_m),w(\rho_m)\}$; the
reverse inequality is trivial. Hence
$\rank_{\F}B_{\rho_m}=\min\{b(\rho_m),w(\rho_m)\}$ and
$\epsilon(\rho_m)=0$.
\end{proof}

\begin{theorem}
\label{thm:staircase}
For every $n\geq 1$, the staircase partition
\[
\rho_n=(n,n-1,\ldots,2,1)
\]
has
\[
\sigma(\rho_n)=\Bigl\lceil \tfrac{n}{2}\Bigr\rceil.
\]
\end{theorem}

\begin{proof}
The case $n=1$ is immediate, since $\rho_1=(1)$ has a single orbit.

Now suppose $n\geq 2$, and put
\[
\mu=\rho_n^{\square}.
\]
By Lemma~\ref{lem:staircaseselfsimilar},
\[
\mu=\rho_{n-2}.
\]
Lemma~\ref{lem:staircaseimbalance} gives
\[
\delta(\mu)
=
\delta(\rho_{n-2})
=
\Bigl\lceil \tfrac{n-2}{2}\Bigr\rceil,
\]
while Lemma~\ref{lem:staircasefullrank} gives
\[
\epsilon(\mu)=0.
\]
Substituting these values into the decomposition of
Proposition~\ref{prop:checkerboardrank}, we obtain
\[
\begin{aligned}
\sigma(\rho_n)
&=
1+\delta(\mu)+2\epsilon(\mu)\\
&=
1+\Bigl\lceil \tfrac{n-2}{2}\Bigr\rceil\\
&=
\Bigl\lceil \tfrac{n}{2}\Bigr\rceil.
\end{aligned}
\]
\end{proof}

\begin{example}
\label{ex:staircase54321}
Take $\rho_5=(5,4,3,2,1)$, the truncation $\rho_7^{\square}$. Of its
fifteen boxes, nine are black and six are white
(Lemma~\ref{lem:staircaseantidiag}), so the minority class ordered as in Lemma~\ref{lem:staircasefullrank} is
\[
M=\{(2,1),(4,1),(1,2),(3,2),(2,3),(1,4)\}.
\]
Applying $\phi(i,j)=(i,j+1)$ gives the
black boxes indexing the columns of $S$, in the same order:
\[
\phi(M)=\{(2,2),(4,2),(1,3),(3,3),(2,4),(1,5)\}.
\]
The resulting submatrix of the biadjacency matrix is:
\[
S=
\begin{pmatrix}
1&0&0&0&0&0\\
0&1&0&0&0&0\\
1&0&1&0&0&0\\
1&1&0&1&0&0\\
1&0&1&1&1&0\\
0&0&1&0&1&1
\end{pmatrix},
\]
lower triangular with unit diagonal, so $\det_{\F}S=1$ and
$\rank_{\F}B_{\rho_5}=6=\min\{9,6\}$; thus $\epsilon(\rho_5)=0$. Since
$\delta(\rho_5)=|9-6|=3=\lceil5/2\rceil$, Theorem~\ref{thm:staircase}
gives
\[
\sigma(\rho_7)=1+\delta(\rho_5)+2\epsilon(\rho_5)=1+3+0=4
=\Bigl\lceil\tfrac{7}{2}\Bigr\rceil.
\]
\end{example}

\begin{remark}
\label{rem:staircasevsrectangle}
Theorem~\ref{thm:staircase} and the rectangular formula
$\sigma(n^m)=\gcd(m,n)$ of Theorem~\ref{thm:gridnullity} occupy opposite
extremes of the decomposition
\[
\sigma(\lambda)-1=\underbrace{\delta(\lambda^{\square})}_{\text{imbalance}}
+\underbrace{2\epsilon(\lambda^{\square})}_{\text{rank deficiency}}.
\]

For a square $\lambda=n^n$ the truncation $\lambda^{\square}=(n-1)^{n-1}$
has equally balanced color classes (for $n$ odd) or almost balanced (for $n$ even), so $\delta \in \{0,1\}$ and the orbit count is carried
entirely or almost entirely by the rank-deficiency term. For the staircase the situation
reverses: $\epsilon\equiv0$ and the orbit count is carried entirely by
the checkerboard imbalance. Thus the staircase is an extremal family for
which the lower bound of Corollary~\ref{cor:checkerboardlowerbound} is
always attained.
\end{remark}
\subsection{An integral refinement}

In this section we show the adjacency matrix nullity we are computing is the same in characteristic zero as over $\F$. A \emph{quadriculated disk} is a topological disk decomposed into
quadrilaterals in such a way that every interior vertex belongs to
exactly four quadrilaterals. In particular, every Young diagram,
viewed as a union of unit squares, is a quadriculated disk. Deift and Tomei \cite{DeiftTomei} proved that when $b(\mu)=w(\mu)$ then the corresponding matrix $B_\mu$ has determinant $0$ or $\pm 1$.

Saldanha and Tomei extended this to arbitrary $b(\mu)$ and $w(\mu)$ and proved a strong integral result about the
black-to-white adjacency matrix of a quadriculated disk
\cite{SaldanhaTomei}. In our notation, their theorem has the following
consequence.

\begin{theorem}[Saldanha--Tomei]
\label{thm:saldanhatomei}
Let $\mu$ be a partition, and regard its biadjacency matrix
$B_\mu$ as a matrix over $\mathbb Z$. If
\[
r=\rank_{\mathbb Q}(B_\mu),
\]
then $B_\mu$ is equivalent over $\mathbb Z$ to a matrix of the
form
\[
\begin{pmatrix}
I_r&0\\
0&0
\end{pmatrix}.
\]
Equivalently, the Smith normal form of $B_\mu$ has $r$ nonzero
entries, all of which are equal to $1$.
\end{theorem}

One immediate consequence is that the rank of $B_\mu$ is independent
of the field:
\[
\rank_K(B_\mu)=\rank_{\mathbb Q}(B_\mu)
\]
for every field $K$. We may therefore replace the rank over
$\mathbb F_2$ in Proposition~\ref{prop:checkerboardrank} by the rank
over any field.

\begin{corollary}
\label{cor:integralorbitformula}
Let $\lambda$ be a partition and put
\[
\mu=\lambda^\square.
\]
For every field $K$,
\[
\sigma(\lambda)
=
1+b(\mu)+w(\mu)-2\rank_K(B_\mu).
\]
Equivalently,
\[
\sigma(\lambda)
=
1+\nul_K A(G_\mu).
\]
\end{corollary}

In particular, the one-orbit condition admits a purely integral
formulation.

\begin{corollary}
\label{cor:oneorbitunimodular}
Let $\lambda$ be a partition and put
\[
\mu=\lambda^\square.
\]
Then
\[
\sigma(\lambda)=1
\]
if and only if
\[
b(\mu)=w(\mu)
\qquad\text{and}\qquad
|\det B_\mu|=1.
\]
Equivalently, $\lambda$ has one billiard orbit if and only if the
two checkerboard color classes of $\mu$ have the same size and
$B_\mu$ is unimodular over $\mathbb Z$.
\end{corollary}

\begin{proof}
By Corollary~\ref{cor:oneorbitcheckerboard}, the condition
$\sigma(\lambda)=1$ is equivalent to $b(\mu)=w(\mu)$ and
$B_\mu$ being nonsingular over $\mathbb F_2$. When $B_\mu$ is
square, Theorem~\ref{thm:saldanhatomei} implies that either
$\det B_\mu=0$ or
\[
|\det B_\mu|=1.
\]
The result follows.
\end{proof}

The same theorem also gives an integral interpretation of the full
orbit number. Since the Smith normal form of $B_\mu$ has only
$1$'s and $0$'s, the same is true of the adjacency matrix
\[
A(G_\mu)
=
\begin{pmatrix}
0&B_\mu\\
B_\mu^{\mathsf T}&0
\end{pmatrix}.
\]
Consequently,
\[
\operatorname{coker}A(G_{\lambda^\square})
\cong
\mathbb Z^{\,\sigma(\lambda)-1}.
\]
Thus the number of billiard orbits is one more than the free rank of
the cokernel of the integer adjacency matrix of
$G_{\lambda^\square}$, and this cokernel has no torsion.

\begin{remark}
\label{rem:simplyconnected}
The integral conclusion of
Theorem~\ref{thm:saldanhatomei} is not restricted to Young diagrams (indeed most of our analysis in this paper can extend beyond just partitions).
Saldanha and Tomei require that every interior vertex belongs to exactly four
quadrilaterals. In particular, their theorem applies to every simply
connected \emph{polyomino}.

Thus the relevant geometric hypothesis is not that the row lengths
form a partition, but rather that the region has no holes, something clearly satisfied in our setting.
The assumption of simple connectivity is critical.
Consider the polyomino obtained from a $4\times 3$ rectangle by
removing its two central squares:
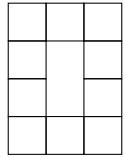
\[
\begin{tikzpicture}[scale=.5, baseline=(current bounding box.center)]
  \draw (0,0) rectangle (1,1);
  \draw (1,0) rectangle (2,1);
  \draw (2,0) rectangle (3,1);

  \draw (0,1) rectangle (1,2);
  \draw (2,1) rectangle (3,2);

  \draw (0,2) rectangle (1,3);
  \draw (2,2) rectangle (3,3);

  \draw (0,3) rectangle (1,4);
  \draw (1,3) rectangle (2,4);
  \draw (2,3) rectangle (3,4);
\end{tikzpicture}
\]
This region has one hole, and its cell-adjacency graph is the cycle
$C_{10}$. Its checkerboard color classes both have five vertices.
With suitable orderings of the black and white cells, its
biadjacency matrix is
\[
B=
\begin{pmatrix}
1&0&1&0&0\\
1&1&0&0&0\\
0&1&0&0&1\\
0&0&1&1&0\\
0&0&0&1&1
\end{pmatrix}.
\]
A direct calculation gives
\[
\det B=2.
\]
Indeed, the Smith normal form of $B$ is
\[
\operatorname{diag}(1,1,1,1,2),
\]
and therefore
\[
\operatorname{coker}(B)
\cong
\mathbb Z/2\mathbb Z.
\]
Thus the torsion-free conclusion of
Theorem~\ref{thm:saldanhatomei} can fail for a polyomino with a hole.

\end{remark}

\subsection{Asymptotically partitions have many billiard orbits}
Looking at Table \ref{tab:orbitdistribution} we see the fraction of all partitions with one orbit appears to be declining toward zero as $n$ grows. Much more is true:

\begin{theorem}
\label{thm:manyorbits}
Fix $d \geq 1$. Let $a(n,d)$ be the number of partitions $\lambda$ of $n$ with $\sigma(\lambda)\leq d$ and $p(n)$ the number of partitions of $n$. Then:
\[
\lim_{n \rightarrow \infty}\frac{a(n,d)}{p(n)}=0.
\]
\end{theorem}
\begin{proof}
Recall that the checkerboard imbalance $\delta(\mu)$ is the absolute
value of the BG-rank of $\mu$. Since $\lambda^\square$ is obtained by
deleting the first row and first column of $\lambda$, the BG-ranks of
$\lambda$ and $\lambda^\square$ differ by at most $1$. Hence
\(
\delta(\lambda)
\leq
\delta(\lambda^\square)+1.
\)
If $\sigma(\lambda)\leq d$, then
Corollary~\ref{cor:checkerboardlowerbound} gives
\(
\delta(\lambda^\square)\leq d-1,
\)
and therefore
\(
\delta(\lambda)\leq d.
\)

Let $p_j(n)$ denote the number of partitions of $n$ having BG-rank
$j$. It follows that
\[
a(n,d)
\leq
\sum_{j=-d}^{d} p_j(n).
\]
For each fixed $j$, the results of Baker and Males~\cite{BakerMales}
give
\[
\lim_{n\to\infty}\frac{p_j(n)}{p(n)}=0.
\]
Since $d$ is fixed, the sum contains only finitely many terms, and hence
\[
0
\leq
\frac{a(n,d)}{p(n)}
\leq
\sum_{j=-d}^{d}\frac{p_j(n)}{p(n)}
\longrightarrow 0.
\]
This proves the result.
\end{proof}

\section{Enumeration of one-orbit partitions}

Let $a(n)$ be the number of partitions $\lambda\vdash n$ with
$\sigma(\lambda)=1$. The odd-tiling criterion gives an exact
generating-function reduction to Young diagrams with an odd number of
domino tilings.

Let $\mathcal D_{\mathrm{odd}}$ be the set of partitions whose Young
diagrams have an odd number of domino tilings, including the empty
partition. For nonempty $\mu$, define
\[
h(\mu)
=
|\mu|+\ell(\mu)+\mu_1+1,
\]
and put
\[
h(\varnothing)=1.
\]

\begin{proposition}
\label{prop:oneorbitgf}
The generating function for one-orbit partitions is
\[
{
\sum_{n\geq1}a(n)q^n
=
\frac{1}{(1-q)^2}
\sum_{\mu\in\mathcal D_{\mathrm{odd}}}
q^{h(\mu)}.}
\]
\end{proposition}
\begin{proof}
If $\mu\neq\varnothing$, the partitions satisfying
$
\lambda^\square=\mu
$
are precisely
\[
\lambda
=
(a,\mu_1+1,\ldots,\mu_{\ell(\mu)}+1,1^t),
\]
where
\[
a\geq\mu_1+1,
\qquad
t\geq0.
\]
Here the middle rows $\mu_i+1$ are forced: they are the rows recovered
by restoring the first column and first row deleted by the
$\square$ operation. Only two features of $\lambda$ are free. The first
row length $a$ may be any integer at least $\mu_1+1$, since it must be
long enough to dominate the row below it; and any number $t\geq0$ of
trailing rows of length one may be appended, as these are erased by the
deletion of the first column and so do not affect $\lambda^\square$.
These two choices vary independently, which is the source of the two
factors below.

The size of such a $\lambda$ is
\[
|\lambda|
=
a+|\mu|+\ell(\mu)+t.
\]
Summing over the two independent parameters $a$ and $t$ factors as a
product of geometric series,
\[
\sum_{a\geq\mu_1+1}\ \sum_{t\geq0}
q^{\,a+|\mu|+\ell(\mu)+t}
=
q^{\,|\mu|+\ell(\mu)}
\underbrace{\Bigl(\sum_{a\geq\mu_1+1}q^{a}\Bigr)}_{=\,q^{\mu_1+1}/(1-q)}
\underbrace{\Bigl(\sum_{t\geq0}q^{t}\Bigr)}_{=\,1/(1-q)}
=
\frac{q^{\,|\mu|+\ell(\mu)+\mu_1+1}}{(1-q)^2}.
\]
The exponent in the numerator is exactly $h(\mu)$. One factor of
$1/(1-q)$ counts the choices of first-row length $a$, the other counts
the number $t$ of appended unit rows.

For $\mu=\varnothing$, the corresponding partitions are the hooks
$(a,1^t)$ with $a\geq1$ and $t\geq0$; the same two independent
parameters contribute
\[
\frac{q}{(1-q)^2}
=
\frac{q^{h(\varnothing)}}{(1-q)^2}.
\]
Summing over all $\mu\in\mathcal D_{\mathrm{odd}}$ and applying
Theorem~\ref{thm:oddtiling}, which identifies the one-orbit partitions
$\lambda$ as those with $\lambda^\square\in\mathcal D_{\mathrm{odd}}$,
completes the proof.
\end{proof}

This identity is exact, but it transfers the principal difficulty to the classification and enumeration of Young diagrams having an odd number of domino tilings, which is not necessarily any easier.

\section{Concluding remarks and open problems}

We have reduced the billiard orbit count $\sigma(\lambda)$ to a series of
increasingly concrete invariants of the cell-adjacency graph: the
dimension of its binary bicycle space, the nullity of its mod-$2$
Laplacian, and finally the nullity of the adjacency matrix of the
truncated diagram $\lambda^{\square}$. This last reduction converts the
one-orbit problem into the parity of a domino-tiling count, and the
checkerboard decomposition
$\sigma(\lambda)-1=\delta(\lambda^{\square})+2\epsilon(\lambda^{\square})$
divides the orbit excess into a  ``color-imbalance'' term and a
``rank-deficiency'' term. Two families sit at the extremes of this
decomposition: rectangles, where the imbalance is as small as possible and the
Fibonacci-polynomial algebra controls the rank deficiency, and
staircases, where the rank deficiency vanishes identically and the
imbalance alone determines the count. Many natural questions remain; we
collect below some that we find most compelling.

For us a central problem is to understand the one-orbit partitions:

\begin{problem}
\label{prob:oneorbit}
Find a simple characterization of the partitions $\lambda$ with
$\sigma(\lambda)=1$, and determine the asymptotic growth of the counting sequence
$a(n)$, the second column of Table~\ref{tab:orbitdistribution}.
\end{problem}

\begin{problem}
\label{prob:zeroforcing}
The argument in the proof of Theorem \ref{thm:durfee-bound} seems to be a special case of a technique in \cite{AIMZeroForcing}, in particular we proved that the vertices corresponding to the end of the first $r$ rows are what this paper calls a \emph{zero forcing set}. Can we find more general examples of zero forcing sets in the graph $G_{\lambda^\square}$ to obtain results on $\sigma(\lambda)$?
\end{problem}

\begin{problem}
\label{prob:fathook}
For rectangular partitions the orbit number is governed by identities
among Fibonacci polynomials over $\F$. Determine whether this algebraic approach extends to
more complicated families. A reasonable first place to start is the family of  ``fat hook'' partitions
$\lambda=(a^{m},b^{n})$ with $a>b$. Can we find an explicit formula for $\sigma(a^m,b^n)$?
\end{problem}

\begin{problem}
\label{prob:reductions}
Determine which transformations of $\lambda$ preserve the orbit number,
and whether each resulting equivalence class contains a canonical
reduced shape. 
\end{problem}

\begin{problem}
\label{prob:holes}
Extend the structural results of this paper from Young diagrams to
arbitrary finite connected unions of unit squares with holes.  The
medial-link and bicycle-space framework of
Section~\ref{sec:graph-theory} continues to apply, and the bounded face
boundaries of the cell-adjacency graph still give a basis of its cycle
space.  Thus an adjacency-nullity formula can be obtained by adjoining
one additional face for each hole.

Determine how these additional ``hole faces'' affect the orbit number
in structural terms.  In particular, characterize the one-orbit
regions, determine how creating or filling a hole changes the orbit
number, and find bounds on this change in terms of the number and
geometry of the holes.  It would also be interesting to understand the
corresponding integral theory.  Remark~\ref{rem:simplyconnected} shows
that, unlike the simply connected case, torsion in the cokernel of the
relevant integer adjacency matrix can occur once holes are present.
\end{problem}

\end{document}